\documentclass[11pt]{amsart}
\usepackage{graphicx}
\usepackage{stmaryrd}
\usepackage{amssymb}
\usepackage{amsthm}
\usepackage{dsfont}
\usepackage{hyperref}
\usepackage[lite,initials,msc-links]{amsrefs}

\newtheorem{theorem}{Theorem}[section]
\newtheorem{proposition}[theorem]{Proposition}
\newtheorem{corollary}[theorem]{Corollary}

\theoremstyle{definition}
\newtheorem*{remark}{Remark}

\numberwithin{equation}{section}

\DeclareMathOperator{\sgn}{sgn}					
\DeclareMathOperator{\tr}{tr}					

\newcommand{\concat}{\oplus}					
\newcommand{\dir}[1]{\overrightarrow{#1}}		
\newcommand{\card}[1]{\lvert{#1}\rvert}			
\newcommand{\ave}[1]{\langle{#1}\rangle}		
\newcommand{\abs}[2][]{\mathopen{#1\lvert}%
 {#2}\mathclose{#1\rvert}}						
\newcommand{\norm}[2][]{\mathopen{#1\lVert}%
 {#2}\mathclose{#1\rVert}}						
\newcommand{\supnorm}[2][]{\mathopen{#1\lVert}%
 {#2}\mathclose{#1\rVert}_\infty}				
\newcommand{\boundary}[1]{{\partial#1}}			
\newcommand{\oddset}[1]{{\delta#1}}				
\newcommand{\bc}{\Box}							
\newcommand{\free}{{\text{\rm free}}}			
\newcommand{\labelmark}{\diamondsuit}			
\newcommand{\Loop}{\ell}						
\newcommand{\Looprep}{{\hat\Loop}}				
\newcommand{\Path}{\wp}							
\newcommand{\lLoop}{\Loop^\labelmark}			
\newcommand{\lLooprep}{{\hat\Loop}^\labelmark}	
\newcommand{\lPath}{\Path^\labelmark}			
\newcommand{\Loops}{\mathcal{L}}				
\newcommand{\Configs}{\mathcal{C}}				
\newcommand{\lConfigs}{\Configs^\labelmark}		
\newcommand{\Decomps}{\mathcal{D}}				
\newcommand{\Pairings}{\mathcal{P}}				
\newcommand{\Id}{\mathrm{I}}					
\newcommand{\eps}{\varepsilon}					
\newcommand{\north}{\mathord\shortuparrow}		
\newcommand{\south}{\mathord\shortdownarrow}	
\newcommand{\east}{\mathord\shortrightarrow}	
\newcommand{\west}{\mathord\shortleftarrow}		
\newcommand{\torus}{\circledcirc}				

\newcommand{\C}{\mathds{C}}						
\newcommand{\R}{\mathds{R}}						
\newcommand{\Z}{\mathds{Z}}						

\begin{document}

\date{July 19, 2013}
\title[Signed loop approach to the Ising model]
	{The signed loop approach to the Ising model:
	 foundations and critical point}
\author{Wouter Kager}
\author{Marcin Lis}
\author{Ronald Meester}
\address{VU University\\
Department of Mathematics\\
De Boelelaan 1081a\\
1081\,HV Amsterdam\\
The Netherlands}
\thanks{The second author (ML) was financially supported by Vidi grant 
639.032.916 of the Netherlands Organisation for Scientific Research (NWO)}
\email{\char`{w.kager,m.lis,r.w.j.meester\char`}\,@\,vu.nl}

\begin{abstract}
	The signed loop approach is a beautiful way to rigorously study the 
	two-dimensional Ising model with no external field. In this paper, we 
	explore the foundations of the method, including details that have so far 
	been neglected or overlooked in the literature. We demonstrate how the 
	method can be applied to the Ising model on the square lattice to derive 
	explicit formal expressions for the free energy density and two-point 
	functions in terms of sums over loops, valid all the way up to the 
	self-dual point. As a corollary, it follows that the self-dual point is 
	critical both for the behaviour of the free energy density, and for the 
	decay of the two-point functions.
\end{abstract}


\maketitle

\section{Introduction and main results}
\label{sec:introduction}

The Ising model~\cite{Ising} was introduced to explain certain properties of 
ferromagnets, in particular, the phenomenon of spontaneous magnetization. In 
1952, Kac and Ward~\cite{KacWard} proposed a method for approaching the Ising 
model on~$\Z^2$, based on studying configurations of signed loops. In this 
self-contained paper, we explore this method in detail and with mathematical 
rigour. This leads to new explicit formal expressions for the free energy 
density and two-point functions in terms of sums over signed loops, which in 
turn allow us to rederive several classical results about the Ising model.

We consider the Ising model on finite rectangles in~$\Z^2$, by which we mean 
graphs $G = (V,E)$ whose vertex sets~$V$ consist of all points of~$\Z^2$ 
contained in a rectangle $[a,b] \times [c,d]$ in~$\R^2$, and whose edge 
sets~$E$ consist of all unordered pairs $\{u,v\}$ with $u,v\in V$ such that 
their $L_1$~distance $\norm{u-v}$ is~1. For brevity, we will also write $uv$ 
instead of $\{u,v\}$ for the edge between $u$ and~$v$ (note that $uv=vu$). The 
\emph{boundary}~$\boundary{G}$ of~$G$ is the set of those vertices~$u$ in~$V$ 
for which there is a $v\in \Z^2\setminus V$ with $\norm{u-v} = 1$.

We associate to~$G$ a space of \emph{spin configurations} $\Omega_G = 
\{-1,+1\}^V$. For $\sigma\in \Omega_G$, $\sigma_v$ denotes the spin at~$v$. 
Sometimes we impose \emph{positive boundary conditions}, meaning that we 
restrict ourselves to the set of configurations
\[
	\Omega_G^+
	= \{\sigma\in\Omega_G \colon \sigma_u = +1 \text{ if $u\in 
	\boundary{G}$}\}.
\]
In contrast, when we speak of \emph{free boundary conditions}, we work with 
the unrestricted set of spin configurations $\Omega_G^\free = \Omega_G$.

The Ising model defines a (Gibbs--Boltzmann) probability distribution on the 
set of spin configurations. At inverse temperature~$\beta$, it is given by
\begin{equation}
	\label{eqn:PIsing}
	P^\bc_{G,\beta}(\sigma)
	= \frac1{Z^\bc_{G,\beta}} \prod_{uv\in E} e^{\beta \sigma_u\sigma_v},
	\qquad \sigma\in \Omega_G^\bc,
\end{equation}
where $\bc\in \{\free,+\}$ stands for the imposed boundary condition, and
$Z^\bc_{G,\beta}$ is the \emph{partition function} of the model, defined as
\begin{equation}
	\label{eqn:ZIsing}
	Z^\bc_{G,\beta}
	= \sum_{\sigma\in\Omega^\bc} \prod_{uv\in E} e^{\beta \sigma_u\sigma_v}.
\end{equation}
To simplify the notation, and following the physics literature, we will 
write
\begin{equation}
	\label{eqn:aveIsing}
	\ave{\,\cdot\,}^\bc_{G,\beta} = E^\bc_{G,\beta}(\,\cdot\,)
\end{equation}
for expectations with respect to~$P^\bc_{G,\beta}$. Important functions are 
the \emph{Helmholtz free energy}
\[
	F^\bc_{G,\beta}	= -\beta^{-1} \ln Z^\bc_{G,\beta}
\]
and the \emph{free energy density}~$f(\beta)$, obtained as the infinite-volume 
limit
\[
	\lim_{G\to\Z^2} \frac1{\card{V}} F^\bc_{G,\beta}
	= \lim_{G\to\Z^2} -\frac1{\beta\card{V}} \ln Z^\bc_{G,\beta}
	=: f(\beta).
\]
It is well known that this limit exists, and also not difficult to show that 
the limit is the same for all boundary conditions, see for 
example~\cite{Simon}*{Section~II.3} (as we shall see, existence of the limit 
actually also follows from the signed loop approach for non-critical~$\beta$). 
The formalism of statistical mechanics predicts that phase transitions 
coincide with discontinuities or other singularities in derivatives of the 
free energy density.

In the case of the Ising model on~$\Z^2$, the classical arguments of 
Peierls~\cite{Peierls} and Fisher~\cite{Fisher} (see also~\cites{Dobrushin, 
Griffiths, GriffithsIII}) established that it does undergo a phase transition, 
which can be characterized in terms of a change in behaviour of the 
infinite-volume limits $\smash{ \ave{ \sigma_u \sigma_v }^\bc_{\Z^2, \beta} }$ 
of the \emph{two-point functions} $\smash{ \ave{ \sigma_u\sigma_v }^\bc_{G, 
\beta} }$. Peierls' argument implies that as $\norm{u-v}\to \infty$, these 
infinite-volume two-point functions stay bounded away from~0 for large 
enough~$\beta$. In contrast, Fisher's argument yields exponential decay to~0 
of these two-point functions as $\norm{u-v}\to\infty$, for sufficiently 
small~$\beta$. However, there is a gap between the two ranges of~$\beta$ for 
which these arguments work, so they are not strong enough to conclude that the 
phase transition is sharp. As we shall see, the signed loop method studied 
here does lead to a proof of this fact.

To be more specific, a signed loop is essentially a closed, non-backtracking 
walk, with a positive or negative weight assigned to it; see 
Section~\ref{ssec:signedloops} for a precise definition. In this paper, we 
demonstrate how the free energy and two-point functions can be expressed as 
(infinite) sums over signed loops in~$\Z^2$, and moreover, how the rate of 
convergence of these sums can be controlled. For instance, 
Theorem~\ref{thm:analyticity} below expresses the free energy 
density~$f(\beta)$ as an explicit formal sum of the weights of all loops for 
which the origin is the ``smallest'' vertex visited (in lexicographic order). 
Likewise, Theorems \ref{thm:lowTcorr} and~\ref{thm:highTcorr} express the 
two-point functions $\smash{ \ave{ \sigma_u \sigma_v }^\bc_{\Z^2,\beta} }$ as 
infinite sums over explicitly defined classes of signed loops.

In all cases, we first derive the corresponding loop expressions for finite 
rectangles~$G$. These were also obtained simultaneously and independently by 
Helmuth~\cite{Helmuth} via the theory of heaps of pieces, but we in addition 
give explicit bounds on the rates of convergence to take the infinite-volume 
limit. As corollaries, without requiring any external results, we rederive 
several classical results about the Ising model, which can be summarized as 
follows:

\begin{corollary}[Sharpness of phase transition]
	\label{cor:sharpness}
	The free energy density~$f(\beta)$ is analytic for all $\beta>0$ except at 
	the self-dual point $\beta = \beta_c$ given by
	\begin{equation}
		\label{eqn:beta_c}
		\exp(-2\beta_c) = \tanh\beta_c = \sqrt2-1.
	\end{equation}
	Moreover, as $\norm{u-v}\to\infty$, $\smash{ \ave{ \sigma_u \sigma_v 
	}^\free_{\Z^2, \beta} }$ decays to~$0$ exponentially fast when $\beta\in 
	(0,\beta_c)$, while $\ave{ \sigma_u\sigma_v }^+_{\Z^2, \beta}$ stays 
	bounded away from~$0$ when $\beta\in (\beta_c,\infty)$.
\end{corollary}

Of course, various other approaches to the Ising model have been developed and 
explored in the past. Most famous are the original algebraic methods of 
Onsager and Kaufman \cites{Kaufman, KauOns, Onsager}, used by them to compute 
the free energy and study correlations (see also~\cite{Palmer}). We also 
mention the approach of Aizenman, Barsky and Fern\'andez~\cite{AizBarFer}, who 
prove sharpness of the phase transition using differential inequalities (this 
approach actually applies to a large class of ferromagnetic spin models in any 
dimension). More recently, in~\cite{BefDum2}, the fermionic observables, 
originally introduced by Smirnov~\cite{Smirnov2} to study the Ising model at 
criticality, have been used in an interesting way for yet another derivation 
of the value of the critical point.

The method considered here (based on the proposition of Kac and Ward) is 
combinatorial in nature, and as such often referred to as the 
\emph{combinatorial method}, but there are other approaches which include 
combinatorial aspects, such as the dimer approach exposed in~\cite{McCoyWu}. 
We therefore prefer to refer to the method considered here as the \emph{signed 
loop approach}.

Over the years, a number of articles developing the signed loop approach have 
appeared in the physics literature, of which the most relevant are 
\cites{Burgoyne, CosMac, Glasser, Sherman1, Sherman3, Vdovichenko1}. However, 
from a mathematical point of view, these papers leave a lot to be desired in 
terms of rigour and technical details. Moreover, doubts have been cast on the 
very validity of the whole method to begin with, not only in the years 
following the Kac--Ward paper, but still recently by Dolbilin et 
al.~\cite{DolShtSht}, who rightly pointed out an error in Vdovichenko's 
article~\cite{Vdovichenko1} (reproduced in~\cite{LanLif}*{\S151}) on the 
signed loop method.

With the present paper, we aim to remove these doubts and deficiencies once 
and for all. To this end, we provide complete, rigorous and detailed proofs of 
the combinatorial identities central to the signed loop method, all in a 
geometric manner. Our proofs of these identities essentially follow the same 
steps as Vdovichenko's paper. A key feature of this particular approach is 
that each configuration of $s$~loops is assigned a signed weight which is 
simply the product of the signed weights of the individual loops, divided 
by~$s!$. This factorization of weights is a crucial aspect of the method, 
which we believe could well be the key to further results beyond this paper. 
We show here that the desired factorization can be made to work if one defines 
the weight of a loop in the right way. Specifically, the error of Vdovichenko 
was that she did not include a loop's \emph{multiplicity} into its weight, as 
we do in equation~\eqref{eqn:loopweight} below. In addition to clarifying the 
signed loop approach and correcting Vdovichenko's error in this way, we also 
apply the results to the Ising model on~$\Z^2$ in ways not considered before 
to derive both new and classical results about the Ising model, as was already 
mentioned above.

The paper is organized as follows. A precise definition of signed loops and 
the formulation of our main results follows in Sections 
\ref{ssec:signedloops}--\ref{ssec:identities}, with our results for the Ising 
model in Section~\ref{ssec:isingmodel}, and our key combinatorial identities 
in Section~\ref{ssec:identities}. A brief discussion of the history and status 
of the combinatorial identities is included at the end of 
Section~\ref{ssec:identities}. The proofs of these identities are given in 
Section~\ref{sec:combinatorial}, and the proofs of our results about the Ising 
model are in Section~\ref{sec:isingresults}.

\subsection{Signed loops}
\label{ssec:signedloops}

Although our applications are in~$\Z^2$, it will be both necessary for this 
paper and of interest for future applications to study signed loops on a more 
general class of graphs. Our starting point is a (finite or infinite) graph $G 
= (V,E)$ embedded in the plane, with vertex set~$V$ and edge set~$E$. We 
identify $G$ with its embedding. We assume $G$ does not have multiple edges, 
but we do not assume that $G$ is planar. For convenience (although this is not 
strictly necessary), we require that edges are straight line segments in the 
embedding, and that except for the vertices at the two endpoints, no other 
vertices lie on an edge. As before, we write $uv$ or~$vu$ for the (undirected) 
edge between $u$ and~$v$.

A \emph{path} of $n$~\emph{steps} in~$G$ is a sequence $(v_0, v_1, \dots, 
v_{n-1}) \in V^n$ such that $v_iv_{i+1}\in E$ for $i = 0, 1, \dots, n-2$, and 
$v_{i+2} \neq v_i$ for $i = 0, 1, \dots,n-3$ (paths are 
\emph{non-backtracking}). If all rotations of the sequence $(v_0, \dots, 
v_{n-1})$ are also paths (so that in particular, $v_0v_{n-1}\in E$), then we 
call the path \emph{closed}. We now order the vertices of~$G$ 
lexicographically by their coordinates in the plane, and define a \emph{loop} 
as a closed path $(v_0, \dots, v_{n-1})$ which is the lexicographically 
smallest element in the collection consisting of all rotations of $(v_0, 
\dots, v_{n-1})$ and all rotations of the reverse sequence $(v_{n-1}, \dots, 
v_0)$ (note that these are all in a way closed paths traversing the same 
loop).

If $\Loop = (v_0,\dots, v_{n-1})$ is a loop or a closed path, we shall make 
the identification $v_j \equiv v_{j\bmod n}$ for all $j\in \Z$. We say that a 
loop~$\Loop$ is \emph{edge-disjoint} if $v_iv_{i+1} \neq v_jv_{j+1}$ for all 
$i,j\in\{0,\dots,n-1\}$ with $i\neq j$. If $\Loop$ is not edge-disjoint, it 
might be the case that the sequence $(v_0,\dots, v_{n-1})$ is periodic, in 
which case we call~$\Loop$ a \emph{periodic loop}. The \emph{multiplicity} 
of~$\Loop$, denoted by~$m(\Loop)$, is its number of steps divided by its 
smallest period. In particular, the multiplicity of every nonperiodic loop 
is~$1$.

\begin{figure}
	\begin{center}
		\includegraphics{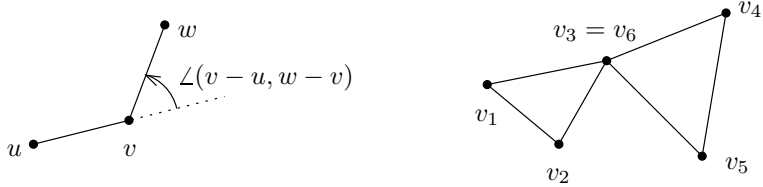}
	\end{center}
	\caption{The turning angle from the vector $v-u$ to the vector $w-v$ 
	(left). The loop $(v_1,v_2,v_3,v_4,v_5,v_6)$ on the right has sign~$-1$, 
	the loop $(v_1,v_2,v_3,v_5,v_4,v_6)$ has sign~$+1$.}
	\label{fig:Angles}
\end{figure}

Given two distinct edges $uv$ and~$vw$, we define $\angle(v-u, w-v) \in 
(-\pi,\pi)$ as the turning angle in the plane from the vector $v-u$ to $w-v$, 
see Figure~\ref{fig:Angles} (left). The \emph{winding angle}~$\alpha(\Loop)$ 
of a loop $\Loop = (v_0, \dots, v_{n-1})$ is simply the sum of all turning 
angles along the loop, that is,
\begin{equation}
	\label{eqn:loopangle}
	\alpha(\Loop) = \sum_{i=0}^{n-1} \angle(v_{i+1}-v_i,v_{i+2}-v_{i+1}).
\end{equation}
We now define the \emph{sign}~$\sgn(\Loop)$ of~$\Loop$ as
\begin{equation}
	\label{eqn:loopsign}
	\sgn(\Loop) = - \exp\Bigl( \frac{i}2 \alpha(\Loop) \Bigr).
\end{equation}
Observe that the winding angle of every loop is a multiple of~$2\pi$ (here we 
use the fact that the edges of~$G$ are straight line segments), hence the sign 
of a loop is either $+1$ or~$-1$.

To define the \emph{signed weight} of a loop, we require a vector $x = 
(x_{uv})_{uv\in E}$ of \emph{edge weights} $x_{uv}\in \R$ (or~$\C$). Given 
these edge weights~$x_{uv}$, the signed weight of a loop $\Loop = 
(v_0,\dots,v_{n-1})$ in~$G$ is defined as
\begin{equation}
	\label{eqn:loopweight}
	w(\Loop; x)
	= \frac{\sgn(\Loop)}{m(\Loop)} \prod_{i=0}^{n-1} x_{v_iv_{i+1}}.
\end{equation}

\begin{remark}
	If a loop is edge-disjoint, it follows from Whitney's 
	formula~\cite{Whitney} that the sign of the loop is $-1$ if and only if 
	the loop crosses itself an odd number of times (see 
	Figure~\ref{fig:Angles}). For loops that are not edge-disjoint, it may not 
	be so clear what is meant by a ``crossing'', but 
	definition~\eqref{eqn:loopsign} makes sense for both kinds of loop. 
	However, if one draws loops in such a way that each visit to an edge is 
	drawn slightly apart from a previous visit, the number of crossings one is 
	forced to draw will always be odd for a loop of sign~$-1$, and even for a 
	loop of sign~$+1$ (see for instance Figure~\ref{fig:Cancellation} 
	below).
\end{remark}

\subsection{Main results for the Ising model}
\label{ssec:isingmodel}

We now return to the Ising model on~$\Z^2$. In this section we will formulate 
our main theorems for the Ising model, which express the free energy density 
and two-point functions in terms of sums over signed loops. Each of our 
theorems will be accompanied by a corollary, which taken together constitute 
Corollary~\ref{cor:sharpness}. Similar results as the ones presented here can 
be obtained for the hexagonal and triangular lattices using the same methods. 
In fact, the method applies to the Ising model on even more general graphs, 
and also allows one to study general $k$-point functions. We intend to go into 
these issues in a subsequent paper.

We start with the free energy density~$f(\beta)$. As we shall prove, 
$f(\beta)$ can be expressed as a sum over those loops in~$\Z^2$ for which the 
origin $o = (0,0)$ is the lexicographically smallest vertex traversed. To be 
more specific, we define $\Loops^\circ_r (\Z^2)$ as the collection of all 
loops $\Loop = (v_0, \dots, v_{r-1})$ of $r$~steps in~$\Z^2$ such that $v_0 = 
o$. We take all edges of~$\Z^2$ to have the same edge weight~$x$. The weights 
$w(\Loop;x)$ of all loops $\Loop\in \Loops^\circ_r (\Z^2)$ are now defined 
by~\eqref{eqn:loopweight}, where by slight abuse of notation, we let $x$ 
denote both the weight of a single edge, and the vector of all edge weights. 
Write
\[
	f^\circ_r(x) = \sum_{\Loop\in \Loops^\circ_r (\Z^2)} w(\Loop;x).
\]

\begin{theorem}
	\label{thm:analyticity}
	The free energy density satisfies
	\begin{equation}
		\label{eqn:freeenergy}
		-\beta f(\beta)	=
		\begin{cases}
			\ln(2\cosh^2\beta) + \sum_{r=1}^\infty f^\circ_r\bigl( \tanh\beta 
			\bigr) & \text{if }\beta\in(0,\beta_c), \\
			2\beta + \sum_{r=1}^\infty f^\circ_r\bigl( \exp(-2\beta) \bigr) & 
			\text{if }\beta\in(\beta_c,\infty).
		\end{cases}
	\end{equation}
\end{theorem}

Note that since $f^\circ_r(x) = f^\circ_r(1) \, x^r$, $\sum_r f^\circ_r(x)$ is 
really a power series in the variable~$x$. The power series 
expressions~\eqref{eqn:freeenergy} show directly that the free energy density 
is an analytic function of~$\beta$ on $(0,\beta_c)\cup(\beta_c,\infty)$. Thus, 
in terms of the behaviour of the free energy density, the Ising model can only 
be critical at the self-dual point~$\beta_c$. That~$f(\beta)$ is not analytic 
at~$\beta_c$ follows from Onsager's formula, which we will obtain as a 
corollary to Theorem~\ref{thm:analyticity}:

\begin{corollary}[Onsager's formula]
	\label{cor:Onsager}
	For $\beta\in (0,\beta_c)$ and for $\beta\in (\beta_c,\infty)$, the free 
	energy density $f(\beta)$ is given by the formula
	\[
		-\frac1\beta \frac1{8\pi^2} \int_0^{2\pi} \!\! \int_0^{2\pi} \ln\bigl[ 
		4\cosh^22\beta - 4\sinh2\beta(\cos\omega_1 + \cos\omega_2) \bigr] \, 
		d\omega_1\,d\omega_2.
	\]
\end{corollary}

The functions $f$ and $u = \partial(\beta f) / \partial\beta$, which is the 
\emph{internal energy density} of the system, are both continuous functions 
of~$\beta$. However, in~\cite{Onsager} Onsager has shown that the 
\emph{specific heat}, which is the derivative of~$u$ with respect to 
temperature, diverges as $\beta$ approaches~$\beta_c$. This shows that 
$\beta_c$ is indeed critical for the behaviour of the free energy.

Next, we look at the magnetic behaviour of the model by considering the 
one-point and two-point functions above and below~$\beta_c$. We start with the 
case $\beta\in (\beta_c,\infty)$. What we will show is that for fixed $u,v\in 
\Z^2$, the functions $\ave{ \sigma_u }^+_{G, \beta}$ and $\ave{ \sigma_u 
\sigma_v }^+_{G,\beta}$ have infinite-volume limits along rectangles~$G$, 
where the limits can be identified in terms of sums over certain classes of 
loops in~$\Z^{2*}$, the dual graph of~$\Z^2$, defined as follows. Given 
$u,v\in \Z^2$, let $\gamma$ be a self-avoiding path in~$\Z^2$ from $u$ to~$v$ 
(see Figure~\ref{fig:2PointGammas}, left). We call a loop in~$\Z^{2*}$ 
\emph{$uv$-odd} if it crosses~$\gamma$ an odd number of times. Similarly, we 
call a loop in~$\Z^{2*}$ \emph{$u$-odd} if it crosses a self-avoiding 
path~$\gamma$ in~$\Z^2$ from $u$ to~$\infty$ an odd number of times. It is not 
difficult to see that neither of these definitions depends on the particular 
choice of~$\gamma$. We write $\Loops^u_r (\Z^{2*})$ and $\Loops^{uv}_r 
(\Z^{2*})$ for the sets of $u$-odd and $uv$-odd loops in~$\Z^{2*}$ of 
$r$~steps, respectively.

Let~$x$ be the vector of edge weights on~$\Z^{2*}$ such that every edge has 
the weight~$\exp(-2\beta)$, and define the weights of loops in~$\Z^{2*}$ 
by~\eqref{eqn:loopweight}. Set
\[
	f^u_r(x) = \sum_{ \Loop\in \Loops^u_r (Z^{2*}) }  w(\Loop; x);
	\qquad
	f^{uv}_r(x) = \sum_{ \Loop\in \Loops^{uv}_r (\Z^{2*}) }  w(\Loop; x).
\]

\begin{theorem}
	\label{thm:lowTcorr}
	For all $\beta\in(\beta_c,\infty)$ and fixed $u,v\in\Z^2$ ($u\neq v$),
	\[\begin{aligned}
		\lim_{G\to\Z^2} \ave{ \sigma_u }^+_{G,\beta}
		= \exp\biggl( -2\sum_{r=1}^\infty f^u_r(x) \biggr)
		&=: \ave{ \sigma_u }^+_{\Z^2,\beta} > 0; \\
		\lim_{G\to\Z^2} \ave{ \sigma_u\sigma_v }^+_{G,\beta}
		= \exp\biggl( -2\sum_{r=1}^\infty f^{uv}_r(x) \biggr)
		&=: \ave{ \sigma_u\sigma_v }^+_{\Z^2,\beta} > 0.
	\end{aligned}\]
\end{theorem}

As a corollary to the proof of this theorem, we will obtain that for $\beta\in 
(\beta_c, \infty)$, the two-point functions in the infinite-volume limit stay 
bounded away from~$0$ as $\norm{u-v}\to \infty$:

\begin{corollary}[Positive two-point functions above~$\beta_c$]
	\label{cor:lowTcorr}
	For all $\beta\in(\beta_c,\infty)$,
	\[
		\lim_{\norm{u-v}\to\infty} \ave{\sigma_u\sigma_v}^+_{\Z^2,\beta}
		= \bigl[ \ave{\sigma_o}^+_{\Z^2,\beta} \bigr]^2 > 0.
	\]
\end{corollary}

\begin{figure}
	\begin{center}
		\includegraphics{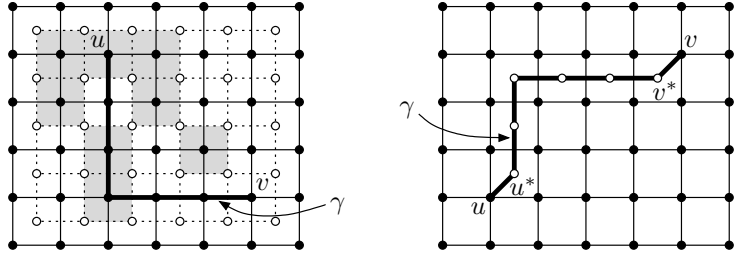}
	\end{center}
	\caption{The paths~$\gamma$ (with bold edges), that we use to study the 
	2-point functions $\ave{\sigma_u \sigma_v}^\bc_{G,\beta}$. The 
	low-temperature case is on the left (spins are~$-1$ in the gray squares, 
	$+1$ in the white regions), the high-temperature case on the right.}
	\label{fig:2PointGammas}
\end{figure}

We now turn to the two-point functions for $\beta\in (0,\beta_c)$. Fix $u,v\in 
\Z^2$, and choose dual vertices $u^*$ and~$v^*$ of~$\Z^{2*}$ such that 
$\norm{u - u^*} = \norm{v - v^*} = 1$. This choice is not unique, but every 
choice of $u^*$ and~$v^*$ will do. Next, choose a self-avoiding path~$\gamma$ 
in~$\Z^{2*}$ from $u^*$ to~$v^*$ (see Figure~\ref{fig:2PointGammas}, right). 
Let $V_\gamma$ denote the set of vertices in~$\gamma$, and let $E_\gamma$ 
denote the union of $\{uu^*, vv^*\}$ with the set of edges traversed 
by~$\gamma$. Write $\Z^2_\gamma$ for the graph obtained from~$\Z^2$ by adding 
the vertices and edges from $V_\gamma$ and~$E_\gamma$ to it.

We define $\Loops^{uu^*}_r (\Z^2_\gamma)$ as the collection of loops 
in~$\Z^2_\gamma$ that visit the edge~$uu^*$ exactly once and have 
$r-\card{E_\gamma}$ steps. Note that by definition, if $\Loop \in 
\Loops^{uu^*}_r (\Z^2_\gamma)$, $r$ only counts the number of steps taken 
by~$\Loop$ along edges of~$\Z^2$, that is, the steps along the edges 
in~$E_\gamma$ are excluded. As our edge weight vector on~$\Z^2_\gamma$ we take 
the vector $x'_\gamma$ such that the weight of every edge in~$E_\gamma$ 
is~$1$, the weight of every edge in~$\Z^2$ which intersects~$\gamma$ is 
$-\tanh\beta$, and the weight of all other edges is~$\tanh\beta$. Set
\[
	f^{uu^*}_r(x'_\gamma)
	= \sum_{ \Loop\in \Loops^{uu^*}_r (\Z^2_\gamma) }  w(\Loop; x'_\gamma),
\]
with $w(\Loop; x'_\gamma)$ defined by~\eqref{eqn:loopweight}. Let $\beta^*$ 
denote the inverse temperature which is dual to~$\beta$, i.e.\ such that 
$\exp(-2\beta^*) = \tanh\beta$.

\begin{theorem}
	\label{thm:highTcorr}
	For all $\beta\in (0,\beta_c)$ and fixed $u,v\in \Z^2$ ($u\neq v$),
	\[
		\lim_{G\to\Z^2} \ave{\sigma_u \sigma_v}^\free_{G,\beta}
		= \biggl( \sum_{r=1}^\infty f^{uu^*}_r (x'_\gamma) \biggr) 
		\ave{\sigma_{u^*} \sigma_{v^*}}^+_{ \Z^{2*}, \beta^* }
		=: \ave{\sigma_u \sigma_v}^\free_{\Z^2,\beta}.
	\]
\end{theorem}

The term $\smash{ \ave{\sigma_{u^*} \sigma_{v^*}}^+_{ \Z^{2*}, \beta^* } }$ 
appearing here is the infinite-volume limit of a two-point function for an 
Ising model on the dual square lattice~$\Z^{2*}$ at the dual inverse 
temperature~$\beta^*$ with positive boundary conditions. It can be expressed 
in terms of signed loops by means of Theorem~\ref{thm:lowTcorr}. As an aside, 
we note that the result in Theorem~\ref{thm:highTcorr} simplifies when $u$ 
and~$v$ are on the same face of~$\Z^2$ (i.e. $\norm{u-v} = 1$), since then we 
can take $u^* = v^*$, so that $\sigma_{u^*} \sigma_{v^*} = 1$. Moreover, since 
the path~$\gamma$ is void in this case, none of the edge weights will be equal 
to~$-\tanh\beta$.

As a corollary to Theorem~\ref{thm:highTcorr} we will obtain that the 
two-point functions decay exponentially to~0 with the distance $\norm{u-v}$ 
for $\beta\in (0,\beta_c)$, which together with Corollaries \ref{cor:Onsager} 
and~\ref{cor:lowTcorr} implies Corollary~\ref{cor:sharpness}:

\begin{corollary}[Decaying two-point functions below~$\beta_c$]
	\label{cor:highTcorr}
	For all $\beta\in(0,\beta_c)$ and fixed $u,v\in \Z^2$, we have that
	\[
		0\leq \ave{\sigma_u \sigma_v}^\free_{\Z^2,\beta} \leq 16 \sum_{r\geq 
		\norm{u-v}} \Bigl( \frac{\tanh\beta}{\tanh\beta_c} \Bigr)^r.
	\]
\end{corollary}

Note that Corollaries \ref{cor:lowTcorr} and~\ref{cor:highTcorr} show 
contrasting behaviour of the two-point functions above and below~$\beta_c$: at 
low temperatures they stay bounded away from~$0$, while for high temperatures 
they decay to~$0$. However, we have used different boundary conditions above 
and below the critical point. Ideally, we would like to use the signed loop 
method to show that for all $\beta\neq\beta_c$, the infinite-volume limit of 
the two-point functions is the same for both boundary conditions. We leave 
this issue for a subsequent paper.

\subsection{Additional edges and loop length}
\label{ssec:additionaledges}

The set of edges~$E_\gamma$ that we introduced above to formulate our 
Theorem~\ref{thm:highTcorr} is an example of what we call \emph{additional 
edges}. As this example shows, we occasionally need these additional edges in 
our applications. They act as ``shortcuts'' that our loops can follow, and in 
general, just as we did above, we do not want to count the steps taken by our 
loops along these shortcuts.

Another example of the use of additional edges is in our proofs leading to 
Corollary~\ref{cor:Onsager}, in which we compare loops in~$\Z^2$ with loops on 
a torus. Here we face a problem, because our methods and theorems about signed 
loops (to be presented below) require that the graph we work on is embedded in 
the plane. As a solution, we will not work on the torus directly, but on a 
representation of it in the plane. As our representation, we take a rectangle 
in~$\Z^2$ with opposite sides connected by additional edges, as illustrated in 
Figure~\ref{fig:Torus}. In this example, the additional edges do not 
correspond to edges that can be traversed by a loop on the torus, and this is 
the reason why steps taken along the additional edges again should not be 
counted.

\begin{figure}
	\begin{center}
		\includegraphics{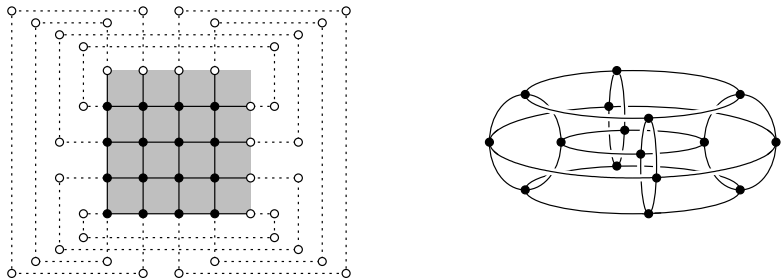}
	\end{center}
	\caption{A square lattice wrapped around a torus (right) and a 
	representation of it in the plane (left). The gray square corresponds to 
	the torus, the dotted lines and open circles are the additional edges and 
	vertices.}
	\label{fig:Torus}
\end{figure}

In general, these considerations lead us to allow the edge set~$E$ of the 
graph $G = (V,E)$ we work on to be divided into a set~$E_A$ of additional 
edges and a set $E\setminus E_A$ of edges that we call \emph{representative}. 
For reasons that will become clear, we must impose that the set~$E_A$ is such 
that the graph $(V,E_A)$ is free of cycles, but otherwise, the edge set can in 
principle be any subset of edges. We now define the \emph{length}~$r(\Loop)$ 
of a loop $\Loop = (v_0,\dots,v_{n-1})$ as the number of~$i$ in $\{0, \dots, 
n-1\}$ such that $v_iv_{i+1} \in E\setminus E_A$. Note the distinction between 
the \emph{length} of a loop and its \emph{number of steps}.

\subsection{The combinatorial identities}
\label{ssec:identities}

We next formulate our combinatorial identities about signed loops for a fixed 
finite graph $G = (V,E)$ embedded in the plane, satisfying the same 
assumptions as in Section~\ref{ssec:signedloops}. In particular, recall that 
edges are straight line segments, and that we do not assume $G$ is planar, 
which implies that two edges can intersect in a point which is not a vertex. 
In this case, we say that the two edges \emph{cross} each other.

We call a subset $F$ of~$E$ \emph{even} if every vertex in the subgraph 
$(V,F)$ of~$G$ has even degree (the empty set is also even). By~$C_F$ we 
denote the total number of unordered pairs of edges in~$F$ that cross each 
other. Given a vector $x = (x_{uv})_{uv\in E}$ of edge weights on~$G$, we now 
define the generating function~$Z(x)$ of even subgraphs of~$G$ as
\begin{equation}
	\label{eqn:Z}
	Z(x)
	= \sum_{\text{even }F\subset E} (-1)^{C_F} \prod_{uv \in F} x_{uv}.
\end{equation}
If the graph~$G$ is planar, we can embed it in such a way that no edges cross 
each other, so that $C_F = 0$ for all~$F$, but in general, an even $F \subset 
E$ may give a negative contribution to the right-hand side of~\eqref{eqn:Z}. 
This makes our generating function different from the one usually studied in 
the literature. Note as a consequence that different embeddings of the same 
(abstract) graph can lead to different functions~$Z(x)$. Since we identify~$G$ 
with its embedding, this last fact does not concern us here.

Our combinatorial identities express the generating function~$Z(x)$ in terms 
of sums over the signed loops in~$G$, with their weights defined 
by~\eqref{eqn:loopweight}. This is what allows us to study the Ising model in 
terms of signed loops, since the free energy and the two-point functions of 
the Ising model can be expressed in terms of graph generating functions, as we 
shall see. Our first identity:

\begin{theorem}
	\label{thm:Zexploopweights}
	For $uv\in E$, let $d_{uv}$ denote the maximum of the degrees of $u$ 
	and~$v$ in the graph~$G$. If $\abs{x_{uv}} < (d_{uv}-1)^{-1}$ for all 
	$uv\in E$, then
	\begin{equation}
		\label{eqn:Zexploopweights}
		Z(x)
		= \exp\biggl( \sum_{\text{$\Loop$ in~$G$}} w(\Loop; x) \biggr).
	\end{equation}
\end{theorem}

We will show that under the condition of Theorem~\ref{thm:Zexploopweights}, 
the loop weights are absolutely summable, so that the order of summation does 
not matter. In particular, let $\Loops_r$ be the collection of all loops of 
length~$r$ in~$G$, and let
\begin{equation}
	\label{eqn:fr(x)}
	f_r(x) = \sum_{\Loop\in\Loops_r} w(\Loop; x).
\end{equation}
Then Theorem~\ref{thm:Zexploopweights} implies that $Z(x)$ equals $\exp\bigl( 
\sum_r f_r(x) \bigr)$, but we claim that this latter equality already holds 
under a significantly weaker condition.

This condition can be formulated in terms of the \emph{transition 
matrix}~$\Lambda(x)$, which we now introduce. If $uv$ is an edge of~$G$, then 
by $\dir{uv}$ we will denote the \emph{directed} edge from $u$ to~$v$. The 
matrix~$\Lambda(x)$ will be indexed by the directed \emph{representative} 
edges of~$G$. Given two directed representative edges $\dir{uv}$ 
and~$\dir{wz}$, we say that $v$ is \emph{linked to}~$w$ if either $v=w$, or 
there exists a sequence of distinct \emph{additional} edges $v_1v_2, v_2v_3, 
\dots, v_{n-1}v_n$ such that $v=v_1$ and $v_n=w$. In the former case, if $v=w$ 
and $u\neq z$, we write
\[
	\angle(\dir{uv}, \dir{wz}) = \angle(v-u, z-w)
\]
for the turning angle from $\dir{uv}$ to~$\dir{wz}$. In the latter case, the 
sequence $(v_1,\dots,v_n)$ is a path (the \emph{chain}) linking $v$ to~$w$, 
passing through additional edges only, and we say that ``$v\leadsto w$ via 
$(v_1,\dots,v_n)$''. By our assumption that the additional edges form no 
cycles, there can be at most one such path. Hence, without ambiguity, if $v$ 
is linked to~$w$ in this way, we can define
\begin{multline*}
	\angle(\dir{uv}, \dir{wz})
	= \angle(v-u, v_2-v) \\\null
	+ \sum_{i=1}^{n-2} \angle(v_{i+1}-v_i, v_{i+2}-v_{i+1})
	+ \angle(w-v_{n-1}, z-w).
\end{multline*}

The transition matrix~$\Lambda(x)$ is now defined as follows. Write 
$\Lambda_{\dir{uv}, \dir{wz}}(x)$ for the entry of the matrix with row index 
$\dir{uv}$ and column index~$\dir{wz}$. Then
\begin{equation}
	\label{eqn:Lambda(x)}
	\Lambda_{\dir{uv},\dir{wz}}(x) = \begin{cases}
		x_{uv} e^{i\angle(\dir{uv},\dir{wz})/2}
		& \text{if $v = w$ and $u\neq z$}; \\
		x_{uv} \prod\limits_{i=1}^{n-1} x_{v_iv_{i+1}} 
		e^{i\angle(\dir{uv},\dir{wz})/2}
		& \text{if $v \leadsto w$ via $(v_1,\dots,v_n)$}; \\
		0 & \text{otherwise}.
	\end{cases}
\end{equation}
Let $\lambda_i(x)$, $i = 1, 2, \dots, 2\card{E\setminus E_A}$, denote the 
eigenvalues of~$\Lambda(x)$, and let $\rho(x) = \max_i \abs{\lambda_i(x)}$ be 
its spectral radius. We will show that if $\rho(x) < 1$, then the $f_r(x)$ are 
absolutely summable. This leads to our second identity, which forms the core 
of the signed loop approach:

\begin{theorem}
	\label{thm:Zdeterminant}
	If $\rho(x) < 1$, then
	\begin{equation}
		\label{eqn:Zdeterminant}
		Z(x)
		= \exp\biggl( \sum_{r=1}^\infty f_r(x) \biggr)
		= \sqrt{\det\bigl( \Id - \Lambda(x) \bigr)}.
	\end{equation}
\end{theorem}

Clearly, to apply Theorem~\ref{thm:Zdeterminant} to the Ising model on~$\Z^2$, 
we will need a bound on the spectral radius~$\rho(x)$. Since $\rho(x)$ is 
bounded from above by the operator norm~$\norm{ \Lambda(x) }$ of $\Lambda(x)$ 
induced by the Euclidean metric, the desired bound is provided by the next 
theorem:

\begin{theorem}
	\label{thm:keybound}
	For a finite rectangle~$G$ in~$\Z^2$ with no additional edges,
	\begin{gather*}
		\norm{\Lambda(x)} \leq (\sqrt2+1) \supnorm{x},\\
		\llap{and\quad}
		\abs{f_r(x)} \leq 2\card{V}r^{-1} (\sqrt2+1)^r \supnorm{x}^r.
	\end{gather*}
\end{theorem}

If we take all edge weights to be~$1$, Theorem~\ref{thm:keybound} says that 
the ``number'' of signed loops of $n$~steps, counted with signs and 
multiplicities included, grows (in absolute value) like $(\sqrt2 + 1)^n$. 
Contrast this with the number of unsigned non-backtracking loops in~$\Z^2$, 
which grows like~$3^n$. It is this reduction in growth rate which allows us to 
go all the way to the critical point, while the classical Peierls and Fisher 
arguments stay far from it. Indeed, in Section~\ref{ssec:isingmodel} we have 
seen that we will take our edge weights to be either $\exp(-2\beta)$ 
or~$\tanh\beta$, so by Theorem~\ref{thm:keybound} and~\eqref{eqn:beta_c}, the 
spectral radius will be smaller than~$1$ for all $\beta\in (\beta_c,\infty)$ 
or all $\beta\in (0,\beta_c)$, respectively.

We conclude this introduction with a few remarks about the history and status 
of the combinatorial identities presented above. Kac and Ward observed in 
their paper~\cite{KacWard} that the Onsager--Kaufman formula for the partition 
function~$Z_{G, \beta}$ of the Ising model on~$\Z^2$ (or rather its square) 
appears to be proportional to the determinant of a matrix~$A_G$, which for 
rectangles in~$\Z^2$ is equivalent to our matrix $\Id-\Lambda(x)$. Various 
attempts were subsequently undertaken to justify the formula $\smash{ 
Z^2_{G,\beta} } \propto \det A_G$, and to rederive the Onsager--Kaufman 
formula in this way. These attempts involved expanding the partition function 
into a formal infinite product over signed loops~\cites{Sherman1, Sherman3, 
Burgoyne}, or a formal infinite sum over signed loop 
configurations~\cite{Vdovichenko1}. In either case, the correct interpretation 
and convergence of the obtained formal expressions are serious mathematical 
issues.

These issues were circumvented by Dolbilin et al.~\cite{DolEtAl} by directly 
comparing the coefficients of the finite polynomials $\smash{ Z^2_{G, \beta} 
}$ and~$\det A_G$, thus rigorously proving the formula $\smash{ Z^2_{G,\beta} 
} \propto \det A_G$ for finite planar graphs~$G$. The same method was then 
employed by Cimasoni to generalize this Kac--Ward formula to graphs embedded 
in surfaces of higher genus~\cite{Cimasoni}. He also exposed a direct relation 
between the Kac--Ward determinant and the adjacency matrix arising in the 
dimer approach.

Historically then, the main focus appears to have been on the equality between 
the extreme left-hand and right-hand sides of 
equation~\eqref{eqn:Zdeterminant}, in cases where $Z(x)$ is proportional to 
the Ising partition function and $A_G = \Id-\Lambda(x)$. For the applications 
to the Ising model presented in this paper, however, the first equality 
in~\eqref{eqn:Zdeterminant} is the more important and relevant one. Therefore, 
our proof proceeds along the lines of the Vdovichenko paper, which is targeted 
at directly expressing~$Z(x)$ as an infinite sum over configurations of signed 
loops. As we go along, we carefully address the issues of interpretation and 
convergence of this sum, mentioned above. In particular, 
Theorem~\ref{thm:keybound}, the key to the convergence issue, is to the best 
of our knowledge a new result.

Moreover, for our applications of the combinatorial results to the Ising 
model, it turns out to be necessary to allow crossing edges. This means that 
our function~$Z(x)$ (and hence also Theorem~\ref{thm:Zdeterminant}) is not 
quite the same as the one considered in the literature so far. In particular, 
our~$Z(x)$ need not be proportional to the Ising partition function for the 
graph~$G$.

\section{Proofs of the combinatorial identities}
\label{sec:combinatorial}

We now turn to the proof of our main Theorems \ref{thm:Zexploopweights} 
and~\ref{thm:Zdeterminant}. Recall that here $G = (V,E)$ is a general finite 
graph embedded in the plane, potentially containing crossing edges or 
additional edges. The proof proceeds in a number of steps. In the first step, 
detailed in Section~\ref{ssec:edgedisjoint}, we will identify each even 
subgraph $(V,F)$ of~$G$ with a number of edge-disjoint collections of loops, 
and show that the sum of their weights yields precisely the contribution 
of~$F$ to the generating function~$Z(x)$. In the second step, in 
Section~\ref{ssec:allloops}, we will explain the conditions under which we can 
express~$Z(x)$ in terms of $\sum_r f_r(x)$ and $\det\bigl( \Id - \Lambda(x) 
\bigr)$, under the assumption that the weights of all remaining configurations 
of loops in~$G$ of total length~$r$ cancel each other. The proof of this 
assumption, the last step, is carried out in 
Section~\ref{ssec:loopcancellation}.

\subsection{Expansion into collections of edge-disjoint loops}
\label{ssec:edgedisjoint}

We will be concerned with crossings of loops and paths, and we need to 
carefully establish the relevant definitions first. More specifically, we will 
consider collections $\{\Loop_1,\dots,\Loop_s\}$ of loops with the properties 
that all loops $\Loop_1,\dots,\Loop_s$ are edge-disjoint, and no two loops in 
the collection visit a common edge. We call these \emph{edge-disjoint 
collections} of loops. Intuitively it may be clear what we mean by a crossing 
of such edge-disjoint loops, but some care is needed, so we will now give the 
precise definitions.

First, consider two paths $(u,v,w)$ and $(x,y,z)$ in~$G$, and let~$A$ be the 
union of the two half-lines $\{v+t(u-v)\colon t\geq0\}$ and $\{v+t(w-v)\colon 
t\geq0\}$. We say that $(u,v,w)$ and $(x,y,z)$ \emph{cross each other at the 
vertex}~$v$ if $v=y$ and the vertices $x$ and~$z$ do not lie in the same 
infinite component of the complement of~$A$ in the plane.

Now let $\Loop_1 = (u_0,\dots, u_{n-1})$ and $\Loop_2 = (v_0, \dots, v_{m-1})$ 
be two loops that form an edge-disjoint pair $\{\Loop_1,\Loop_2\}$. By 
$C_V(\Loop_1, \Loop_2)$ we denote the number of pairs $(i,j)$, where $0\leq 
i<n$ and $0\leq j<m$, such that the paths $(u_{i-1}, u_i, u_{i+1})$ and 
$(v_{j-1}, v_j, v_{j+1})$ cross each other at~$u_i$. We call $C_V(\Loop_1, 
\Loop_2)$ the number of \emph{vertex crossings} between $\Loop_1$ 
and~$\Loop_2$. Similarly, we define the number of \emph{edge crossings} 
between $\Loop_1$ and~$\Loop_2$, denoted $C_E(\Loop_1, \Loop_2)$, as the 
number of pairs $(i,j)$ such that $0\leq i<n$ and $0\leq j<m$, and the edges 
$u_iu_{i+1}$ and $v_jv_{j+1}$ cross each other in~$G$.

We also need to formally define the number of times a loop crosses itself, so 
consider an edge-disjoint loop $\Loop = (v_0, \dots, v_{n-1})$. We define the 
number of \emph{vertex self-crossings} of~$\Loop$, denoted $C_V(\Loop)$, as 
the number of pairs $(i,j)$, where $0\leq i<j<n$, such that $(v_{i-1}, v_i, 
v_{i+1})$ and $(v_{j-1}, v_j, v_{j+1})$ cross each other at the vertex~$v_i$. 
The number of \emph{edge self-crossings} of~$\Loop$, denoted $C_E(\Loop)$, is 
defined as the number of pairs $(i,j)$ such that $0\leq i<j<n$, and the edges 
$v_iv_{i+1}$ and $v_jv_{j+1}$ cross each other in the graph~$G$.

As was already mentioned in the introduction, Whitney's formula~\cite{Whitney} 
says that the sign of an edge-disjoint loop is $-1$ if the loop crosses itself 
an odd number of times, and $+1$ otherwise. In other words, we have
\[
	\sgn(\Loop) = (-1)^{C_V(\Loop) + C_E(\Loop)}
\]
if $\Loop$ is edge-disjoint. We now simply define the sign of an edge-disjoint 
collection of loops $\{\Loop_1,\dots,\Loop_s\}$ as
\begin{equation}
	\label{eqn:sgndisjoint}
	\sgn\{\Loop_1,\dots,\Loop_s\}
	= \prod_{i=1}^s \sgn(\Loop_i)
	= (-1)^{ \sum_{i=1}^s \{C_V(\Loop_i) + C_E(\Loop_i)\} }.
\end{equation}
If $F\subset E$ is even, we can decompose~$F$ into an edge-disjoint collection 
of loops in such a way, that the union of all edges traversed by the loops 
is~$F$ (one way to find such a decomposition is given in the proof of 
Proposition~\ref{pro:edgedisjoint} below). This decomposition is in general 
not unique. We write $\Decomps(F)$ for the set of all possible edge-disjoint 
decompositions of~$F$, and recall that $C_F$ denotes the number of unordered 
pairs of edges in~$F$ that cross each other.

\begin{proposition}
	\label{pro:edgedisjoint}
	For all even subsets~$F$ of~$E$ we have that
	\[
		\sum_{\{\Loop_1,\dots,\Loop_s\} \in \Decomps(F)}
		\sgn\{\Loop_1,\dots,\Loop_s\} = (-1)^{C_F}.
	\]
\end{proposition}

\begin{proof}
	Let $\{\Loop_1, \dots, \Loop_s\}$ be an edge-disjoint collection of loops. 
	Since $\Loop_1, \dots, \Loop_s$ are all closed loops in the plane, any two 
	distinct loops $\Loop_i$ and~$\Loop_j$ from the collection necessarily 
	cross each other an even number of times. That is, $C_V(\Loop_i, \Loop_j) 
	+ C_E(\Loop_i, \Loop_j)$ is even for all $i\neq j$. Therefore,
	\[
		\sgn\{\Loop_1,\dots,\Loop_s\}
		= (-1)^{ \sum_{1\leq i\leq s}\{ C_V(\Loop_i) + C_E(\Loop_i)\} + 
		\sum_{1\leq i<j\leq s} \{ C_V(\Loop_i,\Loop_j) + 
		C_E(\Loop_i,\Loop_j)\} }.
	\]
	Furthermore, if $\{\Loop_1,\dots,\Loop_s\} \in \Decomps(F)$, then clearly 
	the total number of edge crossings occurring among the loops must coincide 
	with~$C_F$, that is,
	\[
		C_F = \sum_{1\leq i\leq s} C_E(\Loop_i) + \sum_{1\leq i<j\leq s} 
		C_E(\Loop_i,\Loop_j).
	\]
	Hence, it suffices to prove that
	\begin{equation}
		\label{eqn:sumvcrossings=1}
		\sum_{\{\Loop_1,\dots,\Loop_s\} \in \Decomps(F)} (-1)^{\sum_{1\leq 
		i\leq s} C_V(\Loop_i) + \sum_{1\leq i<j\leq s} C_V(\Loop_i,\Loop_j)}
		= 1.
	\end{equation}

	\begin{figure}
		\begin{center}
			\includegraphics{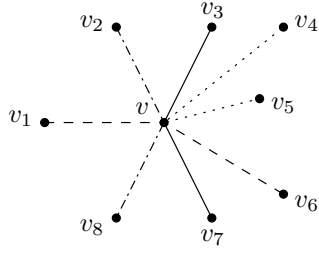}
		\end{center}
		\caption{The neighbours $v_1,v_2,\dotsc$ of a vertex~$v$ in an even 
		subgraph~$(V,F)$ are ordered in a clockwise fashion around~$v$. Two 
		neighbours that are connected by edges drawn in the same line style 
		are paired to each other (see the text).}
		\label{fig:Pairing}
	\end{figure}

	Let $V_F$ denote the set of vertices in~$V$ whose degree in the subgraph 
	$(V,F)$ is nonzero, and let $\deg(v,F)$ denote the degree of~$v$ in 
	$(V,F)$. For $v\in V_F$ of degree~$2k$, write $v_1, \dots, v_{2k}$ for the 
	endpoints of the edges in~$F$ that are incident to~$v$. Assume that these 
	vertices are ordered in a clockwise manner around~$v$, starting from the 
	lexicographically smallest one (see Figure~\ref{fig:Pairing}). Denote by 
	$\Pairings_v(F)$ the collection of partitions of the vertices $v_1, \dots, 
	v_{2k}$ into sets of size~$2$. We call these partitions \emph{pairings at 
	the vertex}~$v$. We write
	\[
		\Pairings(F) = \prod_{v\in V_F} \Pairings_v(F),
	\]
	and call an element of~$\Pairings(F)$ a \emph{pairing} associated with the 
	subgraph~$(V,F)$.

	We have a natural 1--1 correspondence between $\Pairings(F)$ 
	and~$\Decomps(F)$. Indeed, starting from any vertex $v\in V_F$ and any 
	$i\in \{1, \dots, \deg(v,F)\}$, the pairing $\pi\in \Pairings(F)$ defines 
	a unique closed path $(u_0, \dots, u_{n-1})$ with the properties that $u_0 
	= v_i$, $u_1 = v$, and for all~$j$, $u_{j-1}$ is paired with~$u_{j+1}$ at 
	the vertex~$u_j$. Continuing this way, and replacing each closed path 
	obtained by the corresponding loop, yields an edge-disjoint collection 
	$\{\Loop_1, \dots, \Loop_s\}\in \Decomps(F)$. It is easy to see that this 
	defines a bijective relation between $\Pairings(F)$ and~$\Decomps(F)$.

	Using this bijection, we can express the sum 
	in~\eqref{eqn:sumvcrossings=1} equally well as a sum over all pairings. 
	More precisely, for $\pi\in \Pairings(F)$ let $\pi_v$ denote the pairing 
	it induces at the vertex $v\in V_F$, and write $C_v(\pi_v)$ for the number 
	of crossings at the vertex~$v$ introduced by this pairing. We call~$\pi_v$ 
	\emph{even} (\emph{odd}) if $C_v(\pi_v)$ is even (odd). Note that 
	\eqref{eqn:sumvcrossings=1} is equivalent to
	\[
		\sum_{\pi\in \Pairings(F)} (-1)^{\sum_{v\in V_F} C_v(\pi_v)}
		= \prod_{v\in V_F} \sum_{\pi_v\in \Pairings_v(F)} (-1)^{C_v(\pi_v)}
		= 1,
	\]
	from which we see that it suffices to prove that for all $v\in V_F$, the 
	number of even pairings~$\pi_v$ exceeds the number of odd pairings~$\pi_v$ 
	by~1.

	We prove this by induction on the degree~$2k$ of~$v$. Write $N^+_k$ 
	and~$N^-_k$ for the numbers of even, resp.\ odd, pairings for~$v$ of 
	degree~$2k$. For $k=1$ we clearly have $N^+_k = 1$ and $N^-_k = 0$. Now 
	let $k>1$, and suppose that we pair the vertex~$v_1$ with~$v_i$ at~$v$. 
	Next pair the remaining $2k-2$ neighbours~$v_j$ of~$v$ in all possible 
	ways. For even~$i$, there is an even number of~$j$ in between $1$ and~$i$, 
	and therefore the pairing we obtain will be even if and only if the 
	pairing of the remaining $2k-2$ vertices is even (see 
	Figure~\ref{fig:Pairing}). Likewise, for odd~$i$, the obtained pairing 
	will be even if and only if the pairing of the remaining vertices is odd. 
	Since we have $k$ even values for~$i$, and $k-1$ odd values, this gives 
	$N^+_k = kN^+_{k-1} + (k-1)N^-_{k-1}$ and $N^-_k = kN^-_{k-1} + 
	(k-1)N^+_{k-1}$. Hence by the induction hypothesis, $N^+_k-N^-_k = 1$.
\end{proof}

From Proposition~\ref{pro:edgedisjoint} we will now obtain our first main 
result. Recall that
\[
	Z(x) = \sum_{\text{even }F\subset E} (-1)^{C_F} \prod_{uv\in F} x_{uv}.
\]
The case $F = \emptyset$ is treated separately: by convention it 
contributes~$1$ to the sum. Hence, Proposition~\ref{pro:edgedisjoint} implies 
that
\[
	Z(x)
	= 1 +\sum_{\substack{\text{even }F\subset E:\\ F\neq \emptyset}} \sum_{\ 
	\{\Loop_1, \dots, \Loop_s\}\in \Decomps(F)\ } \sgn\{\Loop_1, 
	\dots,\Loop_s\} \prod_{uv\in F} x_{uv}.
\]
Recall that the multiplicity~$m(\Loop)$ of an edge-disjoint loop~$\Loop$ 
is~$1$. Therefore, using \eqref{eqn:sgndisjoint} and the 
definition~\eqref{eqn:loopweight} of the weight of a loop, we can write
\[
	Z(x)
	= 1 + \sum_{\substack{\text{even }F\subset E:\\ F\neq \emptyset}} \sum_{\ 
	\{\Loop_1, \dots, \Loop_s\}\in \Decomps(F)\ } \prod_{i=1}^s w(\Loop_i; x).
\]
Since this is a finite sum, we do not need to worry about the order of 
summation, so we have
\[
	Z(x)
	= 1 + \sum_{r=1}^{\infty} \sum_{\substack{\text{even }F\subset E:\\ 
	\card{F\setminus E_A} = r}} \sum_{\ \{\Loop_1, \dots, \Loop_s\}\in 
	\Decomps(F)\ } \prod_{i=1}^s w(\Loop_i; x).
\]
If we now denote by $\Decomps_r$ the set consisting of all those edge-disjoint 
collections of loops $\{\Loop_1, \dots, \Loop_s\}$ for which the total length 
$\sum_{i=1}^s r(\Loop_i)$ is~$r$, we see that we have established the 
following theorem:

\begin{theorem}
	\label{thm:Zedgedisjoint}
	\[
		Z(x)
		= 1 + \sum_{r=1}^{\infty} \sum_{\{\Loop_1, \dots, \Loop_s\} \in 
		\Decomps_r} \prod_{i=1}^s w(\Loop_i; x).
	\]
\end{theorem}

\subsection{Extension to all loop configurations}
\label{ssec:allloops}

In Theorem~\ref{thm:Zedgedisjoint}, we have expressed the generating 
function~$Z(x)$ as a sum over all edge-disjoint collections of loops in~$G$. 
In this section, we will see that if the edge weights are sufficiently small, 
we can drop the condition that the loops have to be edge-disjoint, and sum 
instead over all possible loop configurations in~$G$. Here, a \emph{loop 
configuration} is simply an ordered sequence $(\Loop_1, \dots, \Loop_s)$ of 
loops; there is no condition that loops have to be edge-disjoint, nor that two 
loops in the configuration have to be distinct (i.e.\ it is allowed that 
$\Loop_i = \Loop_j$ for some $i\neq j$, which is why we work with ordered 
sequences of loops now).

Write $\Configs_r$ for the collection of all loop configurations $(\Loop_1, 
\dots, \Loop_s)$ satisfying $r(\Loop_1) + \dots + r(\Loop_s) = r$. Some of 
these loop configurations will consist of distinct loops that together form an 
edge-disjoint collection of loops. Let $\Configs^*_r$ denote the subset 
of~$\Configs_r$ containing only these edge-disjoint loop configurations. 
Observe that if $\{\Loop_1,\dots,\Loop_s\}$ is an edge-disjoint collection of 
loops, then the corresponding loop configuration $(\Loop_1, \dots, \Loop_s)$ 
has $s!$~permutations. Therefore, by Theorem~\ref{thm:Zedgedisjoint} we 
already have that
\[
	Z(x)
	= 1 + \sum_{r=1}^\infty \sum_{s=1}^\infty \sum_{ (\Loop_1, \dots, \Loop_s) 
	\in \Configs^*_r} \frac1{s!} \prod_{i=1}^s w(\Loop_i; x),
\]
but we claim that here we may sum over~$\Configs_r$ instead of~$\Configs^*_r$:

\begin{theorem}
	\label{thm:Zloopcfgs}
	\[
		Z(x)
		= 1 + \sum_{r=1}^\infty \sum_{s=1}^\infty \sum_{(\Loop_1,\dots,\Loop_s) 
		\in \Configs_r} \frac1{s!} \prod_{i=1}^s w(\Loop_i; x).
	\]
\end{theorem}

Clearly, since $\Configs_r$ is a finite set for every fixed~$r$, this result 
is an immediate consequence of the following proposition:

\begin{proposition}
	\label{pro:loopcancellation}
	For all $r>0$,
	\[
		\sum_{s=1}^\infty \sum_{ (\Loop_1, \dots, \Loop_s) \in \Configs_r 
		\setminus \Configs^*_r} \frac1{s!} \prod_{i=1}^s w(\Loop_i; x) = 0.
	\]
\end{proposition}

The proof of Proposition~\ref{pro:loopcancellation} is involved, and we 
postpone it to Section~\ref{ssec:loopcancellation}. For now, we assume that 
Proposition~\ref{pro:loopcancellation} and hence Theorem~\ref{thm:Zloopcfgs} 
hold, and explain how Theorems \ref{thm:Zexploopweights} 
and~\ref{thm:Zdeterminant} follow from this.

\begin{proof}[Proof of Theorem~\ref{thm:Zexploopweights}]
	By splitting the sum over the set of loop configurations~$\Configs_r$ in 
	Theorem~\ref{thm:Zloopcfgs} according to the lengths of the individual 
	loops, using~\eqref{eqn:fr(x)} we can write
	\begin{align}
		Z(x)
		&= 1 + \sum_{r=1}^\infty \sum_{s=1}^\infty \frac1{s!} 
		\sum_{r_1+\dots+r_s = r} \prod_{i=1}^s \biggl( \sum_{\Loop\in 
		\Loops_{r_i}} w(\Loop; x) \biggr) \notag \\
		&= 1 + \sum_{r=1}^\infty \sum_{s=1}^\infty \frac1{s!} 
		\sum_{r_1+\dots+r_s = r} \prod_{i=1}^s f_{r_i}(x).
		\label{eqn:Zloopcfgs1}
	\end{align}

	Now suppose that, given~$x$, there exist $\gamma\in (0,1)$ and $C<\infty$ 
	such that
	\begin{equation}
		\label{eqn:fcondition}
		\abs{f_r(x)} \leq C \gamma^r \qquad \text{for all~$r$}.
	\end{equation}
	For future reference, we note that this condition is implied by the 
	stronger condition that
	\begin{equation}
		\label{eqn:fabscondition}
		\sum_{\Loop \in \Loops_r} \abs{w(\Loop; x)} \leq C\gamma^r
		\qquad \text{for all~$r$}.
	\end{equation}
	Under condition~\eqref{eqn:fcondition}, if we write $h(r,s)$ for the 
	summand in~\eqref{eqn:Zloopcfgs1}, we have
	\[
		\abs{h(r,s)}
		= \abs[\bigg]{ \frac1{s!} \sum_{r_1+\dots+r_s = r} \prod_{i=1}^s 
		f_{r_i}(x) }
		\leq \frac{C^s}{s!} \binom{r-1}{s-1} \gamma^r,
	\]
	and thus
	\[
		\sum_{s=1}^\infty \sum_{r=1}^\infty \abs{h(r,s)}
		\leq \sum_{s=1}^\infty \frac{C^s}{s!} \sum_{r=s}^\infty 
		\binom{r-1}{s-1} \gamma^r
		= \exp\left( \frac{C\gamma}{1-\gamma} \right)-1.
	\]
	Hence, we can apply Fubini's theorem to interchange the order of summation 
	over $r$ and~$s$ in~\eqref{eqn:Zloopcfgs1}, which yields
	\[
		Z(x)
		= 1 + \sum_{s=1}^\infty \frac1{s!} \sum_{r=1}^\infty 
		\sum_{r_1+\dots+r_s = r} \prod_{i=1}^s f_{r_i}(x).
	\]
	Note that under condition~\eqref{eqn:fcondition}, $\sum_r f_r(x)$ is 
	absolutely convergent. We now apply Mertens' theorem, which says that if a 
	series $\sum_r a_r$ converges absolutely, and the series $\sum_r b_r$ 
	converges, then their Cauchy product converges to $(\sum_r a_r) (\sum_r 
	b_r)$. In particular, by induction, the $s$-fold Cauchy product of the 
	series $\sum_r a_r$ with itself, which is $\sum_r \sum_{r_1+\dots+r_s = r} 
	a_{r_1} a_{r_2} \dotsm a_{r_s}$, converges to $(\sum_r a_r)^s$. Applying 
	this with $a_r = f_r(x)$, we obtain
	\begin{equation}
		\label{eqn:Zexploops}
		Z(x)
		= 1 + \sum_{s=1}^\infty \frac1{s!} \biggl( \sum_{r=1}^\infty f_r(x) 
		\biggr)^s
		= \exp\biggl( \sum_{r=1}^\infty f_r(x) \biggr).
	\end{equation}

	Observe that this result holds already under the weaker of the two 
	conditions \eqref{eqn:fcondition} and~\eqref{eqn:fabscondition}, but that 
	under the stronger condition~\eqref{eqn:fabscondition}, the loop weights 
	can in fact be summed in any order. We will now show that the condition of 
	Theorem~\ref{thm:Zexploopweights} implies~\eqref{eqn:fabscondition}. 
	Indeed, under the condition of Theorem~\ref{thm:Zexploopweights}, there 
	exists $\gamma\in(0,1)$ such that $(d_{uv}-1)|x_{uv}| \leq \gamma$ for all 
	edges $uv\in E$. Observing that if a loop takes a step along~$uv$, then 
	there are at most $d_{uv}-1$ possibilities for the next step, this implies 
	that the sum of $\abs{w(\Loop; x)}$ over all loops~$\Loop$ of $n$~steps is 
	bounded by $\card{V}\gamma^n$. Since a loop of length~$r$ takes at 
	least~$r$ steps, summing over $n\geq r$ 
	yields~\eqref{eqn:fabscondition}.
\end{proof}

\begin{proof}[Proof of Theorem~\ref{thm:Zdeterminant}]
	Recall definition~\eqref{eqn:Lambda(x)} of the entries of~$\Lambda(x)$, 
	which are indexed by the directed representative edges of~$G$. We can 
	interpret this matrix as a transition matrix for non-backtracking paths on 
	the graph~$G'$ which is represented by~$G$. This represented graph~$G'$ 
	can be obtained from~$G$ by removing every chain of additional edges 
	from~$G$, and identifying the two vertices at the ends of this chain (see 
	Figure~\ref{fig:Torus} for an example).
	
	Indeed, consider two directed representative edges $\dir{uv}$ and 
	$\dir{wz}\neq \dir{vu}$ in~$G$, and write $\dir{uv}'$ and~$\dir{wz}'$ for 
	the corresponding directed edges in the represented graph~$G'$. By 
	construction, a non-backtracking path in~$G'$ can make a step from 
	$\dir{uv}'$ to~$\dir{wz}'$ if and only if $v$ is linked to~$w$ in the 
	graph~$G$, since only then will $v$ be identified with~$w$ in~$G'$. This 
	step corresponds to either a direct step from $\dir{uv}$ to~$\dir{wz}$ 
	in~$G$ (if $v=w$), or to a sequence of steps along the chain linking $v$ 
	to~$w$. In either case, the matrix entry $\Lambda_{\dir{uv}, \dir{wz}} 
	(x)$ picks up all edge weights and turning angles associated with these 
	steps in~$G$.

	We can now interpret this entry as describing the weight picked up by a 
	non-backtracking walk in~$G'$ when it steps from $\dir{uv}'$ 
	to~$\dir{wz}'$. Viewed in this way, the entry of the matrix~$\Lambda^r(x)$ 
	indexed by $\dir{uv}$ and~$\dir{wz}$ is equal to the sum of the weights of 
	all non-backtracking paths in~$G'$ of $r$~steps starting from $\dir{uv}'$ 
	and ending on~$\dir{wz}'$. In particular, the sum of the diagonal entries 
	of $\Lambda^r(x)$ is equal to the sum of the weights of all 
	non-backtracking paths in~$G'$ of $r$~steps starting and ending on the 
	same directed edge.

	Now consider a loop~$\Loop$ of length~$r$ in~$G$. Note that it is possible 
	to start traversing this loop from each step it takes along a 
	representative edge in two directions. Mapping the paths thus obtained to 
	the represented graph~$G'$ yields precisely $2r/m(\Loop)$ different 
	non-backtracking paths of $r$~steps in~$G'$ that start and end on the same 
	directed edge. By \eqref{eqn:loopweight}, \eqref{eqn:loopsign} 
	and~\eqref{eqn:fr(x)}, it now follows that
	\[
		\tr \Lambda^r(x) = -2r f_r(x),
	\]
	where the minus sign comes from the minus sign in the 
	definition~\eqref{eqn:loopsign} of the sign of a loop in terms of its 
	winding angle. Expressed in the eigenvalues $\lambda_i(x)$ 
	of~$\Lambda(x)$, we therefore have that
	\begin{equation}
		\label{eqn:feigenvals}
		f_r(x) = -\frac1{2r} \sum_i \lambda_i^r(x).
	\end{equation}

	In particular, condition~\eqref{eqn:fcondition} is satisfied if $\rho(x) = 
	\max_i \abs{ \lambda_i(x) } < 1$, so in this case the same argument as in 
	the proof of Theorem~\ref{thm:Zexploopweights} 
	yields~\eqref{eqn:Zexploops}. Moreover, if $\rho(x) < 1$, then 
	using~\eqref{eqn:feigenvals} we can write
	\[
		Z(x)
		= \exp\biggl( -\frac12 \sum_{r=1}^\infty \sum_i \frac{\lambda_i^r(x)}r 
		\biggr)
		= \exp\biggl( -\frac12 \sum_i \sum_{r=1}^\infty \frac{\lambda_i^r(x)}r 
		\biggr),
	\]
	and since $\sum_{r=1}^\infty u^r/r = -\ln(1-u)$ if $|u|<1$, we conclude 
	that
	\[
		Z(x)
		= \prod_i \bigl( 1-\lambda_i(x) \bigr)^{1/2}
		= \sqrt{\det\bigl( \Id - \Lambda(x) \bigr)}.
		\qedhere
	\]
\end{proof}

\subsection{Cancellation of non-edge-disjoint loop configurations}
\label{ssec:loopcancellation}

We now turn to the missing step in the proofs of Theorems 
\ref{thm:Zexploopweights} and~\ref{thm:Zdeterminant}, which is the proof of 
Proposition~\ref{pro:loopcancellation}. That is, we must show that the weights 
of all loop configurations $(\Loop_1, \dots, \Loop_s)$ which are not 
edge-disjoint and satisfy $r(\Loop_1) + \dots + r(\Loop_s) = r$ for a 
given~$r$, cancel each other. What complicates matters here, is the fact that 
these loop configurations do not cancel each other one by one, see for example 
Figure~\ref{fig:Cancellation}.\footnote{A picture of the same configurations 
appears in~\cite{DolShtSht} to point out the error in Vdovichenko's paper; it 
is crucial here to take the multiplicities of the loops into account.} Our 
strategy of the proof is to map loop configurations to so-called 
\emph{labelled} loop configurations, which do cancel each other one by one, 
and show that this implies cancellation of the unlabelled loop configurations 
for combinatorial reasons.

We will therefore start by introducing the notion of a \emph{labelled} loop, 
and work our way from there towards the notion of a labelled loop 
configuration, and the proof of their cancellation. In words, a labelled loop 
is a loop with a label attached to each step it takes, where the labels are 
distinct positive integers. For periodic loops, the first step is repeated 
after completing a period, and we require that the label of the first step of 
the loop is smaller than the label associated with each of these 
repetitions.

Formally, a labelled loop~$\lLoop$ is a sequence $(v_0, a_0, v_1, a_1, \dots, 
v_{n-1}, a_{n-1})$ satisfying the following conditions:
\begin{enumerate}
	\renewcommand\labelenumi{\theenumi}
	\renewcommand\theenumi{\textbf{L\arabic{enumi}}}
	\item $\Loop = (v_0, \dots, v_{n-1})$ is a loop;
		\label{condition1}
	\item $(a_0, a_1, \dots, a_{n-1})$ is a sequence of distinct positive 
		integers, called the \emph{labelling} of the loop;
		\label{condition2}
	\item if $\Loop$ is periodic, i.e.\ $m(\Loop) > 1$, then $a_0$ is smaller 
		than $a_{kn/m(\Loop)}$ for all $k\in \{1,2,\dots,m(\Loop)-1\}$.
		\label{condition3}
\end{enumerate}
We call the number~$a_i$ the \emph{label on step~$i+1$} of the loop~$\Loop$; 
we also regard it as a label assigned to the edge $v_iv_{i+1}$. We will use 
the superscript~$\scriptstyle \labelmark$ for labelled loops, and the 
unlabelled loop corresponding to a labelled loop will consistently be denoted 
by dropping this superscript: if $\lLoop$ is a labelled loop, then~$\Loop$ is 
the corresponding unlabelled loop, and so on.

\begin{figure}
	\begin{center}
		\includegraphics{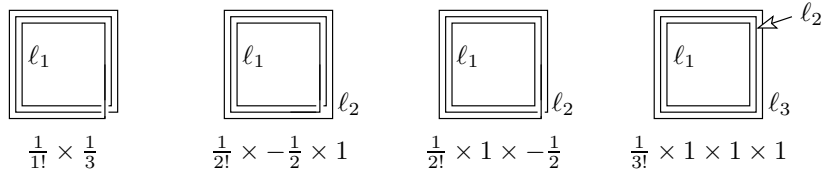}
	\end{center}
	\caption{Four loop configurations on the same vertices and edges, where 
	the traversals of the same edge have been drawn slightly apart to make 
	them discernible. The factors $\frac1{s!} \prod_{i=1}^s \sgn(\Loop_i) / 
	m(\Loop_i)$ are spelled out below each loop configuration to show that the 
	sum of their signed weights is~$0$.}
	\label{fig:Cancellation}
\end{figure}

Observe that one of the effects of labelling loops is that it breaks the 
periodicity of periodic loops: sequences representing labelled loops cannot be 
periodic. Therefore, if~$\Loop$ is periodic, we do not assign to the labelled 
loop $\lLoop = (v_0, a_0, \dots, v_{n-1}, a_{n-1})$ the same weight as to its 
unlabelled counterpart. Instead, we define the weights of labelled loops in 
general by
\begin{equation}
	\label{eqn:lloopweight}
	w(\lLoop; x) = \sgn(\Loop) \prod_{i=0}^{n-1} x_{v_iv_{i+1}},
\end{equation}
where the sign is defined in terms of the winding angle of~$\Loop$ 
by~\eqref{eqn:loopsign}, as before. Note that this weight is actually 
independent of the particular labelling of the loop, and that $w(\lLoop; x) = 
m(\Loop) w(\Loop; x)$.

We write $n(\Loop)$ for the number of steps of a loop~$\Loop$ (recall that this 
is not necessarily the same as the \emph{length}~$r(\Loop)$ of the loop). By a 
\emph{labelled loop configuration} we mean a collection $\{\lLoop_1, \dots, 
\lLoop_s\}$ of labelled loops, in which all labels are distinct and take 
values from the set $\bigl\{ 1, 2, \dots, \sum_{i=1}^s n(\Loop_i) \bigr\}$. In 
particular, any loop configuration $(\Loop_1, \dots, \Loop_s)$ can be turned 
into a labelled loop configuration by attaching a label to every step of every 
loop in such a way, that condition~\ref{condition3} above is fulfilled for 
every labelled loop obtained, and all labels $1, 2, \dots, \sum_{i=1}^s 
n(\Loop_i)$ are used.

Now fix $r$ and~$n$, and consider a loop configuration $(\Loop_1, \dots, 
\Loop_s)$ which is not edge-disjoint and satisfies $\sum_{i=1}^s r(\Loop_i) = 
r$ and $\sum_{i=1}^s n(\Loop_i) = n$. Let~$t$ denote the number of distinct 
loops in $(\Loop_1, \dots, \Loop_s)$, and write $k_1,\dots,k_t$ for the 
respective number of times each of them occurs, so that $k_1 + \dots + k_t = 
s$. Consider the collection of all labelled loop configurations $\{\lLoop_1, 
\dots, \lLoop_s\}$ that can be obtained from $(\Loop_1, \dots, \Loop_s)$ by 
labelling the loops, as described above. For a periodic loop~$\Loop_i$, only 
one of the rotations of its labelling, rotated over a multiple of the smallest 
period, satisfies condition~\ref{condition3}. Furthermore, interchanging the 
labellings of two identical loops $\Loop_i$ and~$\Loop_j$ ($\Loop_i = \Loop_j$ 
but $i\neq j$) yields the same labelled loop configuration. Therefore, the 
number of labelled loop configurations we obtain from $(\Loop_1, \dots, 
\Loop_s)$ is precisely
\[
	\frac{n!}{\prod_{i=1}^s m(\Loop_i) \prod_{i=1}^t k_i!}.
\]

We assign to each of these labelled loop configurations the same weight 
$\prod_{i=1}^s w(\lLoop_i; x)$, where we use the fact that according to the 
definition~\eqref{eqn:lloopweight}, $w(\lLoop_i; x)$ does not depend on the 
actual labelling. Then the total weight of all labelled loop configurations 
associated with $(\Loop_1, \dots, \Loop_s)$ is
\[
	\frac{n!}{\prod_{i=1}^s m(\Loop_i) \prod_{i=1}^t k_i!} \prod_{i=1}^s 
	w(\lLoop_i; x)
	= \frac{n!}{\prod_{i=1}^t k_i!} \prod_{i=1}^s w(\Loop_i; x).
\]
We claim that this is exactly $n!$ times the total weight that all the 
permutations of the loop configuration $(\Loop_1, \dots, \Loop_s)$ contribute 
to the sum in Proposition~\ref{pro:loopcancellation}. Indeed, there are 
precisely
\[
	\frac{s!}{\prod_{i=1}^s k_i!}
\]
such permutations, and the weight each of them contributes to the sum is
\[
	\frac1{s!} \prod_{i=1}^s w(\Loop_i; x).
\]

We conclude that to prove Proposition~\ref{pro:loopcancellation}, it suffices 
to show that for given $n$ and~$r$, the weights of all labelled loop 
configurations $\{\lLoop_1, \dots, \lLoop_s\}$ such that $\sum_{i=1}^s 
n(\Loop_i) = n$, $\sum_{i=1}^s r(\Loop_i) = r$ and $(\Loop_1, \dots, \Loop_s)$ 
is not edge-disjoint, sum to~0. Write $\lConfigs_{n,r}$ for this collection of 
labelled loop configurations. We will now prove the desired cancellation of 
weights, and hence Proposition~\ref{pro:loopcancellation}, by finding a 
bijection $g\colon \lConfigs_{n,r} \to \lConfigs_{n,r}$ which maps each 
labelled loop configuration to a labelled loop configuration which has a 
weight of the opposite sign, but with the same absolute value.

\begin{proof}[Proof of Proposition~\ref{pro:loopcancellation}]
	Before we go into the formal details of the bijection, let us give an 
	informal description of how it will work. Consider a labelled loop 
	configuration $\{\lLoop_1, \dots, \lLoop_s\} \in \lConfigs_{n,r}$, and let 
	$E^\labelmark$ be the set of edges in~$G$ that are assigned more than~1 
	label in this configuration. Find the smallest of all the labels that are 
	assigned to the edges in~$E^\labelmark$, let $a$ be this label, and let 
	$uv$ be the edge to which this label is assigned. Next, find the second 
	smallest label~$b$ which is assigned to the edge~$uv$.

	The label~$a$ labels a step of one of the loops~$\Loop_i$. The label~$b$ 
	either labels another step of the same loop~$\Loop_i$, or it labels a step 
	of a second loop~$\Loop_j$, $i\neq j$. The bijection involves 
	interchanging the ``connections'' on one side of the two steps marked $a$ 
	and~$b$ (either at the vertex~$u$ or at the vertex~$v$), as illustrated in 
	Figure~\ref{fig:Bijection}. It is clear that this operation does not 
	change the absolute value of the weight of the configuration, since the 
	total number of steps that go through a given edge does not change. But 
	Figure~\ref{fig:Bijection} also suggests that the operation corresponds to 
	increasing or decreasing the number of ``crossings'' in the configuration 
	by~1, which should indeed lead to a change in sign.
	
	However, signs were formally defined in terms of winding angles, not 
	numbers of crossings, since it is more difficult to make sense of the 
	latter when loops are not edge-disjoint. Furthermore, we must still 
	formally define the mapping~$g$. We will now deal with these technical 
	issues.

	\begin{figure}
		\begin{center}
			\includegraphics{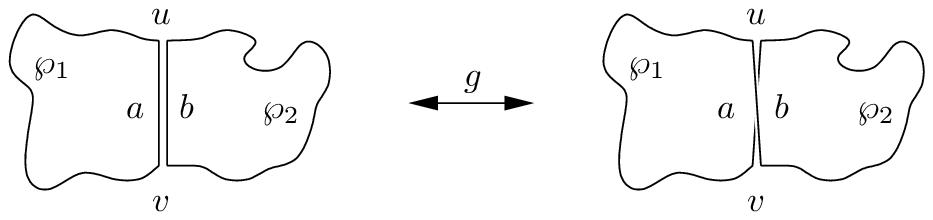}\\
			\includegraphics{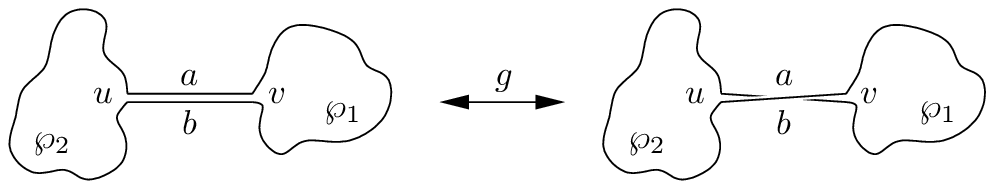}
		\end{center}
		\caption{All cases that occur in the cancellation of labelled loop 
		diagrams, as explained in the text. The curves $\Path_1$ and~$\Path_2$ 
		represent arbitrary paths connected to the vertices $u$ and~$v$.}
		\label{fig:Bijection}
	\end{figure}

	For the formal treatment of the bijection, we need to introduce some 
	additional notation. Given a sequence $a = (a_0, \dots, a_n)$ of arbitrary 
	elements, we write $a^{-1}$ for its \emph{reversion} $a^{-1} = (a_n, 
	a_{n-1}, \dots, a_0)$. If $b = (b_0, \dots, b_m)$ is another sequence of 
	arbitrary elements, we write $a \oplus b$ for the concatenation of $a$ 
	with~$b$, that is,
	\[
		a \concat b = (a_0, \dots, a_n, b_0, \dots, b_m).
	\]

	The weight of a labelled loop configuration $\{\lLoop_1, \dots, 
	\lLoop_s\}$ is defined as the product of the signs of the loops $\Loop_1, 
	\dots, \Loop_s$, times the product of all the edge weights picked up by 
	all the loops. As was anticipated above, the product of edge weights will 
	not change under the bijection, so we will only be concerned with the 
	product of the signs of the loops. We recall from \eqref{eqn:loopangle} 
	and~\eqref{eqn:loopsign} that the sign of a loop $\Loop = (v_0, \dots, 
	v_{n-1})$ is defined in terms of its winding angle as
	\begin{equation}
		\label{eqn:rloopsign}
		\sgn(\Loop)
		= - \exp\Bigl( \frac{i}2 \alpha(\Loop) \Bigr),
	\end{equation}
	where the winding angle $\alpha(\Loop)$ is given by
	\begin{equation}
		\label{eqn:rloopangle}
		\alpha(\Loop)
		= \sum_{i=0}^{n-1} \angle(v_{i+1}-v_i, v_{i+2}-v_{i+1}).
	\end{equation}

	We now define the winding angle and sign of a closed path $(v_0, \dots, 
	v_{n-1})$ by the exact same formulas. In particular, all rotations of a 
	loop~$\Loop$ have the same winding angle and sign. On the other hand, the 
	reversion~$\Loop^{-1}$ of~$\Loop$ and all its rotations are traversed in 
	the opposite direction, and therefore they all have winding angle 
	$\alpha(\Loop^{-1}) = -\alpha(\Loop)$. However, since the winding angle of 
	a loop is a multiple of $2\pi$, we do have that
	\begin{equation}
		\label{eqn:signreversion}
		\sgn(\Loop^{-1}) = \sgn(\Loop)
		\qquad \text{for all closed paths~$\Loop$}.
	\end{equation}
	We call all the rotations of a loop~$\Loop$, and all rotations of its 
	reversion~$\Loop^{-1}$, alternative \emph{representations} of~$\Loop$. All 
	these representations have the same sign. Likewise, the rotations of a 
	labelled loop~$\lLoop$ and its reversion~$\smash{(\lLoop)}^{-1}$ will be 
	called \emph{representations} of this labelled loop.

	We also need to define the winding angle for paths in~$G$ which are not 
	loops. Note that a path $\Path = (v_0, \dots, v_{n-1})$ is not a loop if 
	$v_0v_{n-1} \notin E$, $v_0 = v_{n-2}$, or $v_1 = v_{n-1}$. If we follow 
	such a path from $v_0$ to~$v_{n-1}$, we turn through $n-2$ angles, and it 
	is natural to define the winding angle of~$\Path$ by
	\[
		\alpha(\Path)
		= \sum_{i=0}^{n-3} \angle(v_{i+1}-v_i, v_{i+2}-v_{i+1}).
	\]

	We now have all the notation we need to define and analyse the bijection 
	formally. So consider a labelled loop configuration $\{\lLoop_1, \dots, 
	\lLoop_s\} \in \lConfigs_{n,r}$, and define $E^\labelmark$, $a$, $b$ and 
	$uv$ as above. We will now explain to which labelled loop configuration 
	our configuration $\{\lLoop_1, \dots, \lLoop_s\}$ is mapped by the 
	bijection, and prove that the image has the opposite sign, and hence the 
	opposite weight. There are three possible cases to consider, which are 
	illustrated in Figure~\ref{fig:Bijection}.

	\textbf{Case 1:} \textit{The labels $a$ and~$b$ belong to different 
	labelled loops.} Let $\lLoop_i$ be the labelled loop containing label~$a$, 
	and let $\lLoop_j$ be the labelled loop containing label~$b$. Then these 
	labelled loops have representations of the form $\lLooprep_i = (u,a,v) 
	\concat \lPath_1$ and $\lLooprep_j = (u,b,v) \concat \lPath_2$, 
	respectively, where $\lPath_1$ and~$\lPath_2$ are paths interspersed with 
	labels. We can now form the combined representation
	\[
		\lLooprep_{ij}
		= (u,a,v) \concat \lPath_1 \concat (u,b,v) \concat \lPath_2
	\]
	of a new labelled loop~$\lLoop_{ij}$. Our bijection maps $\{\lLoop_1, 
	\dots, \lLoop_s\}$ to the labelled loop configuration
	\[
		\{\lLoop_1, \dots, \lLoop_s, \lLoop_{ij}\}
		\setminus \{\lLoop_i, \lLoop_j\}.
	\]
	To see that this labelled loop configuration has the opposite sign of its 
	pre-image $\{\lLoop_1, \dots, \lLoop_s\}$, note that by 
	\eqref{eqn:rloopsign}--\eqref{eqn:signreversion},
	\[\begin{split}
		\sgn(\Loop_i)\sgn(\Loop_j)
		&= \sgn(\Looprep_i)\sgn(\Looprep_j)
		 = \exp\Bigl( \frac{i}2 \alpha(\Looprep_i) + \frac{i}2 
		 \alpha(\Looprep_j) \Bigr) \\
		&= \exp\Bigl( \frac{i}2 \alpha(\Looprep_{ij}) \Bigr)
		 = -\sgn(\Looprep_{ij})
		 = -\sgn(\Loop_{ij}).
	\end{split}\]

	\textbf{Case 2:} \textit{The labels $a$ and~$b$ are on steps of the same 
	labelled loop taken in the same direction.} This case is the reverse of 
	Case~1. The labels $a$ and~$b$ are in a labelled loop~$\lLoop_i$ which has 
	a representation of the form
	\[
		\lLooprep_i
		= (u,a,v) \concat \lPath_1 \concat (u,b,v) \concat \lPath_2.
	\]
	From this we obtain the representations $(u,a,v) \concat \lPath_1$ and 
	$(u,b,v) \concat \lPath_2$ of two new labelled loops $\lLoop_{i1}$ 
	and~$\lLoop_{i2}$. The bijection maps $\{\lLoop_1, \dots, \lLoop_s\}$ to 
	the labelled loop configuration
	\[
		\{\lLoop_1, \dots, \lLoop_s, \lLoop_{i1}, \lLoop_{i2}\}
		\setminus \{\lLoop_i\}.
	\]
	The same argument as in Case~1 shows that $\sgn(\Loop_{i1}) 
	\sgn(\Loop_{i2}) = - \sgn(\Loop_i)$.

	\textbf{Case 3:} \textit{The labels $a$ and~$b$ are on steps of the same 
	labelled loop taken in opposite directions.} In this case the labels $a$ 
	and~$b$ are in a labelled loop~$\lLoop_i$ which has a representation of 
	the form
	\[
		\lLooprep_i
		= (u,a,v) \concat \lPath_1 \concat (v,b,u) \concat \lPath_2.
	\]
	From this we can construct the representation
	\[
		\lLooprep
		= (u,a,v) \concat \smash{(\lPath_1)}^{-1} \concat (v,b,u) \concat 
		\lPath_2
	\]
	of a new labelled loop~$\lLoop$. The bijection maps $\{\lLoop_1, \dots, 
	\lLoop_s\}$ to the labelled loop configuration
	\[
		\{\lLoop_1, \dots, \lLoop_s, \lLoop\} \setminus \{\lLoop_i\}.
	\]

	To verify that these loop configurations have opposite signs, observe that
	\begin{equation}
		\label{eqn:angles1}
		\alpha(\Looprep_i)
		= \alpha\bigl( (u,v) \concat \Path_1 \concat (v,u) \bigr)
		+ \alpha\bigl( (v,u) \concat \Path_2 \concat (u,v) \bigr),
	\end{equation}
	and likewise
	\begin{equation}
		\label{eqn:angles2}
		\alpha(\Looprep)
		= \alpha\bigl( (u,v) \concat \Path_1^{-1} \concat (v,u) \bigr)
		+ \alpha\bigl( (v,u) \concat \Path_2 \concat (u,v) \bigr),
	\end{equation}
	where $\Path_1$ and~$\Path_2$ are the paths obtained from $\lPath_1$ 
	and~$\lPath_2$ by dropping the labels. Now notice that upon reversion,
	\begin{equation}
		\label{eqn:angles3}
		\alpha\bigl( (u,v) \concat \Path_1 \concat (v,u) \bigr)
		= -\alpha\bigl( (u,v) \concat \Path_1^{-1} \concat (v,u) \bigr).
	\end{equation}
	Furthermore, it is not difficult to see that
	\[
		\alpha\bigl( (u,v) \concat \Path_1 \concat (v,u) \bigr)
		 = 2m\pi + \pi \qquad \text{for some $m\in\Z$}.
	\]
	Together with \eqref{eqn:angles1}, \eqref{eqn:angles2} 
	and~\eqref{eqn:angles3}, this implies
	\[
		\frac{\sgn(\Loop_i)}{\sgn(\Loop)}
		= \frac{\sgn(\Looprep_i)}{\sgn(\Looprep)}
		= \exp\Bigl( \frac{i}2 \alpha(\Looprep_i) - \frac{i}2 \alpha(\Looprep) 
		\Bigr)
		= -1.
	\]

	We conclude that in all cases, the labelled loop configuration 
	$\{\lLoop_1, \dots, \lLoop_s\}$ is mapped to a labelled loop configuration 
	of opposite weight. From the explicit descriptions given above, it is not 
	difficult to see that the mapping is bijective. As we have explained 
	above, this implies Proposition~\ref{pro:loopcancellation}.
\end{proof}

\section{Proofs of our results for the Ising model}
\label{sec:isingresults}

In this section, we will apply Theorem~\ref{thm:Zdeterminant} to the Ising 
model on the square lattice~$\Z^2$. This will lead to explicit expressions for 
the free energy density and two-point functions in terms of sums over loops 
in~$\Z^2$ or its dual~$\Z^{2*}$, valid all the way up to the critical point. 
We start with a brief review of the low- and high-temperature expansions in 
Section~\ref{ssec:expansions}. The bound on the operator norm in 
Theorem~\ref{thm:keybound} will be derived in Section~\ref{ssec:eigenvalues}. 
Then we will study the free energy density in Section~\ref{ssec:freeenergy}, 
and finally the two-point functions at low and high temperatures in Sections 
\ref{ssec:lowT} and~\ref{ssec:highT}, respectively.

\subsection{Low- and high-temperature expansions}
\label{ssec:expansions}

The partition function of the Ising model is closely related to the graph 
generating function~$Z(x)$ from Section~\ref{ssec:identities}. This can be 
seen from the \emph{low-} and \emph{high-temperature} expansions considered in 
this section. More details on these expansions and the related duality of the 
Ising model can be found in~\cite{Simon}*{Section~II.7}.

Let $G = (V,E)$ be a finite rectangle in~$\Z^2$. By $G^* = (V^*,E^*)$ we shall 
denote the \emph{weak dual} graph of~$G$, i.e.\ the rectangle in~$\Z^{2*}$ 
whose vertices are the centres of the faces of~$G$ (see 
Figure~\ref{fig:Expansions}, left). For our purposes, the low-temperature 
expansion is best considered in the case of positive boundary conditions. It 
is not difficult to see that in this case, there is a 1--1 correspondence 
between the even subgraphs of the weak dual~$G^*$ and the spin configurations 
in~$\Omega^+$: given $\sigma\in \Omega^+$, one obtains the corresponding even 
subset $F(\sigma)$ of~$E^*$ by including the edge dual to $uv$ in $F(\sigma)$ 
if and only if $\sigma_u \neq \sigma_v$, for every $uv\in E$. See 
Figure~\ref{fig:Expansions} (left) for an illustration.

Note that by this correspondence, if $F\subset E^*$ is even, then every edge 
in~$F$ separates two spins that have opposite sign. This means that adding an 
edge $uv$ to~$F$ decreases $\sigma_u\sigma_v$ from $+1$ to~$-1$, and hence has 
a ``cost'' $\exp(-2\beta)$ in the probability distribution~\eqref{eqn:PIsing}. 
It follows that we can write
\begin{equation}
	\label{eqn:lowTmeasure}
	P^+_{G,\beta}(\sigma)
	= \frac{\exp(\beta\card{E})}{Z^+_{G,\beta}} \prod_{uv\in F(\sigma)} 
	x_{uv}, \qquad \sigma\in \Omega_G^+,
\end{equation}
where $x_{uv} = \exp(-2\beta)$ for every $uv\in E^*$, and
\begin{equation}
	\label{eqn:lowTexpansion}
	Z^+_{G,\beta} = \exp(\beta\card{E}) \sum_{\text{even }F\subset E^*} 
	\prod_{uv\in F} x_{uv}.
\end{equation}

This is the low-temperature expansion of the partition function for positive 
boundary conditions. Observe that up to the factor $\exp(\beta\card{E})$, this 
expansion takes exactly the form~\eqref{eqn:Z} of the graph generating 
function~$Z(x)$ for the dual graph~$G^*$ (in which no edges cross each other), 
if we set the edge weights of all dual edges $uv\in E^*$ equal to $x_{uv} = 
\exp(-2\beta)$.

\begin{figure}
	\begin{center}
		\includegraphics{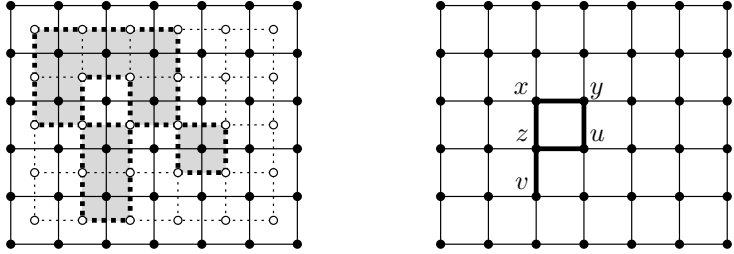}
	\end{center}
	\caption{Left: the graph~$G$ and its weak dual~$G^*$, with an even 
	subgraph of~$G^*$ marked by bold dashed edges. The spins in the gray 
	squares have value~$-1$, the rest have value~$+1$. Right: the subgraph 
	of~$G$ drawn with bold edges contributes $\sigma_x^2 \sigma_y^2 \sigma_z^3 
	\sigma_u^2 \sigma_v (\tanh\beta)^5$ in the high-temperature expansion.}
	\label{fig:Expansions}
\end{figure}

We now turn to the high-temperature expansion, for which we impose free 
boundary conditions. The expansion will be over even subgraphs of the 
graph~$G$, rather than of the dual~$G^*$. Unlike in the low-temperature 
expansion, these subgraphs do not have a clear geometric interpretation, so we 
will take some time to explain how they arise.

The high-temperature expansion starts from equation~\eqref{eqn:ZIsing} and the 
observation that $\sigma_u\sigma_v$ can only take the values $-1$ or~$+1$. 
Since $\exp(\pm\beta) = \cosh\beta \pm \sinh\beta$, we see that
\[
	Z^\free_{G,\beta}
	= (\cosh\beta)^{\card{E}} \sum_{\sigma \in \Omega_G^\free} \prod_{uv\in E} 
	\bigl(1 + \sigma_u\sigma_v \tanh\beta \bigr).
\]
The next step is to expand the product over $uv\in E$. Each term in the 
expansion will be a product of factors obtained by choosing for each edge~$uv$ 
whether $1$~is taken as a factor, or $\sigma_u\sigma_v \tanh\beta$, so that 
the expansion becomes a sum over all choices of factors for each edge~$uv$. We 
can represent each choice graphically by removing the edge~$uv$ if we choose 
the factor~$1$ for this edge, and keeping~$uv$ if we choose the factor 
$\sigma_u\sigma_v \tanh\beta$. This gives a 1--1 correspondence between all 
terms in the expansion, and all $F\subset E$ (not just the even ones). See 
also Figure~\ref{fig:Expansions} (right).

Using this correspondence, and then interchanging the order of summation over 
$\sigma$ and~$F$, we may now write the partition function as
\[
	Z^\free_{G,\beta}
	= (\cosh\beta)^{\card{E}} \sum_{F\subset E} \sum_{\sigma\in\Omega_G^\free} 
	\prod_{u\in V} \sigma_u^{\deg(u,F)} \prod_{uv\in F} x_{uv},
\]
where $x_{uv}=\tanh\beta$ for all $uv\in E$ and $\deg(u,F)$ denotes the degree 
of~$u$ in the graph $(V,F)$. Note that the sum over~$\sigma$ vanishes unless 
$\deg(u,F)$ is even for all $u\in V$, in which case the sum yields simply 
$2^{\card{V}}$. Therefore,
\begin{equation}
	\label{eqn:highTexpansion}
	Z^\free_{G,\beta}
	= 2^{\card{V}} (\cosh\beta)^{\card{E}} \sum_{\text{even }F\subset E} 
	\prod_{uv\in F} x_{uv}.
\end{equation}
Again, up to a multiplicative constant, the expansion takes exactly the 
form~\eqref{eqn:Z} of the graph generating function~$Z(x)$, this time for the 
graph~$G$, if we set the edge weights equal to $x_{uv} = \tanh\beta$.

\subsection{Bound on the operator norm}
\label{ssec:eigenvalues}

Let $G = (V,E)$ be a fixed finite rectangle in~$\Z^2$ with no additional edges 
(i.e.\ $E_A = \emptyset$). Without loss of generality, we may assume that the 
vertex set is
\begin{equation}
	\label{eqn:vertexset}
	V = \{0,1,\dots,M-1\} \times \{0,1,\dots,N-1\}.
\end{equation}
Since we are on the square lattice, directed edges can point in only 
4~directions, and we now introduce some convenient notation for this specific 
case. We write $v\north$, $v\south$, $v\east$ and $v\west$ for the directed 
edges from~$v$ to, respectively, $v+(0,1)$, $v-(0,1)$, $v+(1,0)$ and 
$v-(1,0)$. We also write $\north v$ for the directed edge pointing from 
$v-(0,1)$ to~$v$, and define $\south v$, $\east v$, $\west v$ analogously.

Given a vector of edge weights $x = (x_{uv})_{uv\in E}$ on~$E$, $\Lambda(x)$ 
is the transition matrix indexed by the directed edges of~$G$, defined 
by~\eqref{eqn:Lambda(x)}. For a vertex~$v$ not on the boundary of~$G$, the row 
of~$\Lambda(x)$ indexed by $\east v$, for instance, has exactly~3 nonzero 
entries, corresponding to the 3~possible steps that a loop can take 
from~$\east v$. To be precise, with $u = v-(1,0)$, these 3~entries are
\[
	\Lambda_{\east v,v\east}(x) = x_{uv}, \quad
	\Lambda_{\east v,v\north}(x) = x_{uv} e^{i\pi/4}, \quad
	\Lambda_{\east v,v\south}(x) = x_{uv} e^{-i\pi/4}.
\]
Observe that most rows of~$\Lambda(x)$ have exactly 3~nonzero entries. The 
only exceptions are the rows indexed by directed edges pointing to a vertex 
in~$\boundary{G}$. These exceptional rows make it impossible to compute the 
eigenvalues of~$\Lambda(x)$ directly. We will therefore make the graph 
periodic by connecting opposite sides, as described in 
Section~\ref{ssec:additionaledges}, so that all vertices can be treated alike, 
and then bound the eigenvalues of~$\Lambda(x)$ in terms of those of the 
periodic graph (or equivalently, a graph wrapped on a torus).

To be precise, we first extend our graph~$G$ to a graph~$G^\torus$ ($\torus$ 
stands for ``torus''), by adding edges and vertices as shown in 
Figure~\ref{fig:Torus} (left). Note that this adds directed 
\emph{representative} edges $v\east$ and $\west v$ for every vertex~$v$ on the 
right boundary of~$G$, and $v\north$  and $\south v$ for~$v$ on the top 
boundary. All other edges that are added are considered as \emph{additional} 
edges in the graph~$G^\torus$. Henceforth, when we work on the 
graph~$G^\torus$, computations will be performed modulo $M$ and~$N$ in the two 
respective lattice directions.

We define~$\Lambda^\torus$ as the transition matrix for the graph~$G^\torus$, 
with specific edge weights chosen as follows: all representative edges 
of~$G^\torus$ have edge weight~$1$; for the additional edges, we choose the 
edge weights in such a way, that the product of the edge weights along every 
chain of additional edges linking opposite sides of the rectangle to each 
other is~$-1$. Note that by this choice, the factor~$-1$ will exactly 
compensate the sign picked up by a path which follows the chain, because of 
the 4~quarter-turns it makes.

\begin{proof}[Proof of Theorem~\ref{thm:keybound}]
	We first prove that the operator norm of the matrix~$\Lambda^\torus$ is 
	$\sqrt2+1$. To this end, assume that the rows of~$\Lambda^\torus$ are 
	arranged in such a way, that for every vertex $v\in V$, the 4~rows indexed 
	by $\east v$, $\north v$, $\west v$ and~$\south v$ immediately succeed 
	each other in this order. Let $\Pi$ be the permutation matrix which 
	permutes the columns of~$\Lambda^\torus$ so that column~$vd$ maps to 
	column~$dv$, for all $v\in V$ and $d \in \{\north, \south, \east, 
	\west\}$.

	By construction, the matrix $\Lambda^\torus \Pi$ with the permuted columns 
	is now a block-diagonal matrix, since the 4~rows indexed by the directed 
	edges pointing to~$v$ are matched along the diagonal with the 4~columns 
	indexed by the directed edges pointing out from~$v$. By considering the 
	turning angles, it is easy to see that each $4\times4$ block is equal to 
	the Hermitian matrix
	\[
		A = \begin{bmatrix}
			1 & \exp(i\pi/4) & 0 & \exp(-i\pi/4) \\
			\exp(-i\pi/4) & 1 & \exp(i\pi/4) & 0 \\
			0 & \exp(-i\pi/4) & 1 & \exp(i\pi/4) \\
			\exp(i\pi/4) & 0 & \exp(-i\pi/4) & 1
		\end{bmatrix},
	\]
	which has eigenvalues $\sqrt2+1$ and $\sqrt2-1$, both of multiplicity~2.

	Since~$A$ is Hermitian, its spectral radius is equal to its operator 
	norm~$\norm{A}$. It follows that the operator norm of $\Lambda^\torus \Pi$ 
	is given by $\norm{A} = \sqrt2+1$, and since permuting columns does not 
	change the operator norm of a matrix, we conclude that 
	$\norm{\Lambda^\torus} = \sqrt2+1$.

	We will now use this fact, together with the sub-multiplicativity of the 
	operator norm, to bound~$\norm{ \Lambda(x) }$. To this end, let $D(x)$ be 
	the diagonal matrix of the same dimensions as~$\Lambda^\torus$, defined as 
	follows. For vertices~$v$ on the right boundary of~$G$, the diagonal 
	entries of~$D(x)$ on the rows $v\east$ and~$\west v$ are~$0$, and so are 
	the diagonal entries on the rows $v\north$ and~$\south v$ for $v$ on the 
	top boundary of~$G$. For all other directed edges $\dir{uv}$ in the 
	graph~$G^\torus$, the diagonal entry of~$D(x)$ on row~$\dir{uv}$ is equal 
	to the edge weight~$x_{uv}$.

	Now consider the matrix $D(x) \Lambda^\torus D(1)$, where~$1$ denotes the 
	edge weight vector on~$G^\torus$ with constant weight~$1$ on every edge. 
	The multiplication by~$D(x)$ multiplies all rows of~$\Lambda^\torus$ 
	corresponding to directed edges~$\dir{uv}$ in the graph~$G$ by~$x_{uv}$, 
	and zeroes out all rows corresponding to directed edges which are in the 
	graph~$G^\torus$, but not in the graph~$G$. The multiplication by~$D(1)$ 
	then zeroes out all columns of~$\Lambda^\torus$ corresponding to directed 
	edges which are in~$G^\torus$ but not in~$G$. In other words, $D(x) 
	\Lambda^\torus D(1)$ is just the matrix $\Lambda(x)$ with rows and columns 
	of zeros added to it for every directed edge which is in~$G^\torus$ but 
	not in~$G$. Therefore,
	\[
		\norm{\Lambda(x)}
		= \norm{D(x) \Lambda^\torus D(1)}
		\leq \norm{D(x)} \cdot \norm{\Lambda^\torus} \cdot \norm{D(1)}
		= (\sqrt2+1) \supnorm{x}.
	\]
	The desired bound on $\abs{f_r(x)}$ now follows 
	from~\eqref{eqn:feigenvals} and the facts that $\rho(x) \leq \norm{ 
	\Lambda(x) }$ and the number of directed edges in~$G$ is bounded 
	by~$4\card{V}$.
\end{proof}

\begin{remark}
	Using Fourier transforms, we can actually compute all eigenvalues 
	of~$\Lambda^\torus$, and show that its spectral radius is $\sqrt2+1$. 
	Also, with some extra effort, it is possible to show that for all finite 
	rectangles, the spectral radius of~$\Lambda(x)$ is \emph{strictly} less 
	than $(\sqrt2+1) \supnorm{x}$.
\end{remark}

\subsection{Free energy density}
\label{ssec:freeenergy}

We are now going to use the bound obtained in Theorem~\ref{thm:keybound} to 
prove Theorem~\ref{thm:analyticity} and it's Corollary~\ref{cor:Onsager}.

\begin{proof}[Proof of Theorem~\ref{thm:analyticity}]
	We start with the high-temperature case, so fix $\beta\in (0,\beta_c)$ and 
	set $x = \tanh\beta$. Let $G$ be a rectangle in~$\Z^2$, and take the set 
	of additional edges to be empty. Note that by~\eqref{eqn:beta_c}, 
	$x\in(0,\sqrt2-1)$. By~\eqref{eqn:highTexpansion},
	\begin{equation}
		\label{eqn:ZfreeZx}
		\ln Z^\free_{G,\beta}
		= \card{V}\ln2 + \card{E} \ln(\cosh\beta) + \ln Z_G(x),
	\end{equation}
	where $Z_G(x)$ is the generating function for the graph~$G$ with edge 
	weights equal to~$x$. By Theorems \ref{thm:Zdeterminant} 
	and~\ref{thm:keybound}, $\ln Z_G(x)$ equals $\sum_r f_{G,r}(x)$, where 
	$f_{G,r}(x)$ denotes the sum of the weights of all loops of length~$r$ 
	in~$G$.

	Consider these loops of length~$r$ in~$G$. For each vertex $v\in V$, let 
	$\Loops^v_r(G)$ denote the collection of loops in~$G$ of length~$r$ for 
	which $v$ is the smallest vertex traversed. Observe that if $v$ has 
	distance at least~$r$ to the boundary of~$G$, then $\Loops^v_r(G)$ can be 
	mapped bijectively to~$\Loops^\circ_r (\Z^2)$ by a translation on~$\Z^2$, 
	hence $\sum_{\Loop\in \Loops^v_r(G)} w(\Loop;x) = f^\circ_r(x)$. There are 
	at most $\card{\boundary{G}} r$ vertices at a distance less than~$r$ from 
	$\boundary{G}$, and since $\abs{x}<1$, for such a vertex~$v$ we have 
	$\sum_{\Loop\in \Loops^v_r(G)} \abs{w(\Loop;x)} \leq 3^r$ (by counting 
	non-backtracking paths). From these observations and the fact that 
	$\lim_{G\to\Z^2} \card{\boundary{G}} / \card{V} = 0$, it follows that
	\[
		\lim_{G\to\Z^2} \frac1{\card{V}} f_{G,r}(x)
		= \lim_{G\to\Z^2} \frac1{\card{V}} \sum_{v\in V} \sum_{\Loop\in 
		\Loops^v_r(G)} w(\Loop;x)
		= f^\circ_r(x) \quad \text{for all $r\geq1$}.
	\]
	Furthermore, Theorem~\ref{thm:keybound} says that $\abs{f_{G,r}(x)} \leq 
	2\card{V}r^{-1} (\sqrt2+1)^r x^r$. Therefore, by dominated convergence and 
	Theorem~\ref{thm:Zdeterminant},
	\[
		\lim_{G\to\Z^2} \frac1{\card{V}} \ln Z_G(x)
		= \lim_{G\to\Z^2} \sum_{r=1}^\infty \frac1{\card{V}} f_{G,r}(x)
		= \sum_{r=1}^\infty f^\circ_r(x).
	\]

	We now combine this with~\eqref{eqn:ZfreeZx}, and use $\lim_{G\to\Z^2} 
	\card{E} / \card{V} = 2$ to obtain
	\[
		-\beta f(\beta)
		= \lim_{G\to\Z^2} \frac1{\card{V}} \ln Z^\free_{G,\beta}
		= \ln(2\cosh^2\beta) + \sum_{r=1}^\infty f^\circ_r(x).
	\]
	The low-temperature case can be treated in a similar manner, except that 
	one must work on the dual graphs~$G^*$ with edge weights $x = 
	\exp(-2\beta)$ on the dual edges, and use \eqref{eqn:lowTexpansion} 
	instead of~\eqref{eqn:highTexpansion}.
\end{proof}

\begin{proof}[Proof of Corollary~\ref{cor:Onsager}]
	We will now show that Onsager's formula follows from the expressions 
	for~$f(\beta)$ derived above. First, we claim that with $x=\tanh\beta$ for 
	$\beta\in (0,\beta_c)$ and $x=\exp(-2\beta)$ for $\beta\in 
	(\beta_c,\infty)$, we have
	\begin{equation}
		\label{eqn:freeenergy2}
		-\beta f(\beta) = \ln\left[ \frac{2\cosh2\beta}{1+x^2} \right] + 
		\sum_{r=1}^\infty f^\circ_r(x)
	\end{equation}
	for all these~$\beta$. This follows from \eqref{eqn:freeenergy} and the 
	equality $\cosh^2\beta + \sinh^2\beta = \cosh^2\beta(1+x^2) = \cosh2\beta$ 
	for $\beta\in (0,\beta_c)$, and from \eqref{eqn:freeenergy} together with 
	the logarithm of the equality $2\cosh2\beta = e^{2\beta}(1+x^2)$ for 
	$\beta\in (\beta_c,\infty)$.

	In the proof of Theorem~\ref{thm:analyticity}, we have obtained $\sum_r 
	f^\circ_r(x)$ as the limit of $\card{V}^{-1} \sum_r f_{G,r}(x)$. It is 
	clear from the proof that here we may as well replace $f_{G,r}(x)$ by the 
	corresponding sum of loop weights for the periodic graph~$G^\torus$ from 
	Section~\ref{ssec:eigenvalues}. In fact, the argument becomes even simpler 
	on~$G^\torus$, since we no longer have to treat vertices near the boundary 
	separately. The transition matrix generating the loops in~$G^\torus$ with 
	the desired edge weights~$x$ is~$x\Lambda^\torus$. Hence, by 
	Theorem~\ref{thm:Zdeterminant} and the proof of Theorem~\ref{thm:keybound} 
	in Section~\ref{ssec:eigenvalues}, which gives $\norm{ x\Lambda^\torus } 
	\leq x(\sqrt2+1)$, we see that for $x\in (0,\sqrt2-1)$,
	\begin{equation}
		\label{eqn:computingf}
		\sum_{r=1}^\infty f^\circ_r(x)
		= \lim_{G\to\Z^2} \frac1{\card{V}} \sum_{r=1}^\infty f_{G^\torus,r}(x)
		= \lim_{G\to\Z^2} \frac1{\card{V}} \frac12 
		\ln\det(\Id-x\Lambda^\torus).
	\end{equation}

	We can compute $\det(\Id-x\Lambda^\torus)$ by taking the Fourier transform 
	of~$\Lambda^\torus$. This computation has appeared in the literature 
	before, see for instance~\cites{Glasser, Sherman1, Vdovichenko1}, but we 
	also present it in brief form here for completeness.
	
	Without loss of generality, we may assume that~$V$ is the 
	set~\eqref{eqn:vertexset}, in which case the Fourier transform 
	of~$\Lambda^\torus$ is defined as
	\[
		\widetilde\Lambda^\torus_{(p,q)d,(p',q')d'}
		= \frac1{MN} \sum_{k,k'=0}^{M-1} \sum_{l,l'=0}^{N-1}
		e^{-\frac{2\pi i}M(pk-p'k')-\frac{2\pi i}N(ql-q'l')}
		\Lambda^\torus_{(k,l)d,(k',l')d'},
	\]
	where $d,d'\in \{\north, \south, \east, \west\}$. The calculation of this 
	Fourier transform is made straightforward by the periodicity 
	of~$\Lambda^\torus$, and reveals that the only entries surviving the 
	summations are those for which $p'=p$ and $q'=q$. Hence, 
	$\widetilde\Lambda^\torus$ is a block-diagonal matrix of $4\times4$ 
	blocks. To be precise, writing $\omega_p = 2\pi p/M$, $\omega_q = 2\pi 
	q/N$, the $4\times4$ block for given $p$ and~$q$ is
	\[
		\widetilde\Lambda^\torus_{(p,q)\cdot,(p,q)\cdot}
		= \begin{bmatrix}
			e^{i\omega_p}& e^{i\omega_p+i\pi/4}& 0& e^{i\omega_p-i\pi/4} \\
			e^{i\omega_q-i\pi/4}& e^{i\omega_q}& e^{i\omega_q+i\pi/4}& 0 \\
			0& e^{-i\omega_p-i\pi/4}& e^{-i\omega_p}& e^{-i\omega_p+i\pi/4} \\
			e^{-i\omega_q+i\pi/4}& 0& e^{-i\omega_q-i\pi/4}& e^{-i\omega_q}
		\end{bmatrix}.
	\]

	Since $\det(\Id-x\Lambda^\torus) = \det(\Id-x\widetilde\Lambda^\torus)$, 
	from this Fourier transform we obtain
	\[\begin{split}
		\det(\Id-x\Lambda^\torus)
		&= \prod_{p=0}^{M-1} \prod_{q=0}^{N-1} \det\Bigl( 
		\Id-x\widetilde\Lambda^\torus_{(p,q)\cdot,(p,q)\cdot} \Bigr) \\
		&= \prod_{p=0}^{M-1} \prod_{q=0}^{N-1} \bigl[ (1+x^2)^2 - 2x(1-x^2) 
		(\cos\omega_p + \cos\omega_q) \bigr].
	\end{split}\]
	Using~\eqref{eqn:computingf}, we conclude that
	\begin{multline}
		\label{eqn:sumfcirc}
		\sum_{r=1}^\infty f^\circ_r(x)
		= \lim_{M,N\to\infty} \frac1{2MN} \ln \det(\Id-x\Lambda^\torus) \\
		= \frac1{8\pi^2} \int_0^{2\pi} \!\! \int_0^{2\pi} \ln \bigl[ (1+x^2)^2 
		- 2x(1-x^2) (\cos\omega_1 + \cos\omega_2) \bigr]\, 
		d\omega_1\,d\omega_2.
	\end{multline}
	To finish the computation, note that by~\eqref{eqn:freeenergy2}, we have
	\[
		-\beta f(\beta) = \frac1{8\pi^2} \int_0^{2\pi} \!\! \int_0^{2\pi} 
		\ln\left[ \frac{4\cosh^2 2\beta}{(1+x^2)^2} \right] \, d\omega_1\, 
		d\omega_2 + \sum_{r=1}^\infty f^\circ_r(x).
	\]
	Combining this with~\eqref{eqn:sumfcirc}, and then using the identity
	\[
		\frac{2x(1-x^2)}{(1+x^2)^2} = \frac{\sinh2\beta}{\cosh^22\beta},
	\]
	which holds both for $x = \exp(-2\beta)$ and for $x = \tanh\beta$, we 
	obtain
	\[
		-\beta f(\beta)
		= \frac1{8\pi^2} \int_0^{2\pi} \!\! \int_0^{2\pi} \ln \bigl[ 
		4\cosh^22\beta - 4\sinh2\beta (\cos\omega_1 + \cos\omega_2) \bigr] 
		\,d\omega_1\,d\omega_2.
	\]
	This is Onsager's formula for the isotropic Ising model on~$\Z^2$.
\end{proof}

\subsection{Low-temperature correlations}
\label{ssec:lowT}

In this section, we discuss the Ising model with positive boundary conditions. 
We consider rectangles $G = (V,E)$ in~$\Z^2$ (which later tend to~$\Z^2$) and 
denote by~$G^*$ the weak dual of~$G$. Recall that every spin configuration 
$\sigma\in \Omega_G^+$ on~$G$ corresponds bijectively to an even subgraph 
of~$G^*$, that is, a graph in which all vertices in~$V^*$ have even degree. 
For given~$\sigma$, we denote the corresponding even subset of~$E^*$ 
by~$F(\sigma)$; for given even $F\subset E^*$, we denote the corresponding 
spin configuration by~$\sigma(F)$.

Setting $x_e=e^{-2\beta}$ for every edge~$e$ in~$\Z^{2*}$, by 
\eqref{eqn:lowTmeasure} and~\eqref{eqn:lowTexpansion} we have
\begin{equation}
	\label{eqn:PlowT}
	P_{G,\beta}^+(\sigma)
	= \frac1{Z_{G^*}(x)} \prod_{e\in F(\sigma)} x_e,
	\qquad \sigma\in \Omega_G^+,
\end{equation}
where $Z_{G^*}(x)$ is the generating function for~$G^*$ with edge weight 
vector $x = (x_e)_{e\in E^*}$. Note that here we implicitly restrict the edge 
weight vector~$x$ on~$\Z^{2*}$ to the edges of the graph~$G^*$ we work on. 
Such implicit restrictions to the relevant edges will occur throughout this 
and the following section.

\begin{proof}[Proof of Theorem~\ref{thm:lowTcorr}]
	Fix $u,v \in\Z^2$, $u\neq v$, and let $\gamma$ be a self-avoiding path 
	in~$\Z^2$ from $u$ to~$v$. We may assume that $G$ is large enough so that 
	$u$, $v$ and~$\gamma$ are all contained in the area spanned by~$G$, see 
	Figure~\ref{fig:2PointGammas}. We will express the two-point function 
	$\ave{\sigma_u \sigma_v}^+_{G, \beta}$ as the quotient of two generating 
	functions. To this end, we define new edge weights~$x'_e$ on the edges 
	of~$\Z^{2*}$ such that $x'_e = -x_e$ if $e$ crosses~$\gamma$, and $x_e' = 
	x_e$ otherwise. The reason for defining the weights~$x_e'$ in this way is 
	the crucial fact that for all $\sigma\in \Omega_G^+$,
	\begin{equation}
		\label{eqn:lowTwhygamma}
		\sigma_u \sigma_v \prod_{e \in F(\sigma)} x_e
		= \prod_{e \in F(\sigma)} x_e'.
	\end{equation}
	To see this, recall that the edges in~$F(\sigma)$, by their very 
	definition, cross edges $xy\in E$ for which $\sigma_x\neq \sigma_y$. If 
	$\sigma_u = \sigma_v$, then following~$\gamma$ from $u$ to~$v$, we 
	necessarily cross an even number of such edges. If $\sigma_u\neq 
	\sigma_v$, then we cross an odd number. In either case, 
	\eqref{eqn:lowTwhygamma} holds. 

	With the help of~\eqref{eqn:PlowT}, we can write
	\[
		\ave{\sigma_u\sigma_v}^+_{G,\beta}
		= \sum_{\sigma\in \Omega_G^+} \sigma_u\sigma_v P^+_{G,\beta}(\sigma)
		= \frac1{Z_{G^*}(x)} \sum_{\sigma\in \Omega_G^+} \sigma_u\sigma_v  
		\prod_{e\in F(\sigma)} x_e,
	\]
	and using \eqref{eqn:lowTwhygamma} and the bijection between even~$F$ 
	and~$\Omega_G^+$, we obtain
	\begin{equation}
		\label{eqn:lowTaveRatio}
		\ave{\sigma_u\sigma_v}^+_{G,\beta}
		= \frac1{Z_{G^*}(x)} \sum_{\text{even }F\subset E^*} \prod_{e\in F} 
		x_e'
		= \frac{Z_{G^*}(x')}{Z_{G^*}(x)}.
	\end{equation}
	The idea that correlations in the Ising model can be studied by means of 
	ratios of generating functions with changed edge weights (or equivalently, 
	changed spin-spin interactions) has arisen before in the physics 
	literature, see~\cite{KadCeva}. It now follows from Theorems 
	\ref{thm:Zdeterminant} and~\ref{thm:keybound} that for $\beta > \beta_c$, 
	we have
	\[
		\ave{ \sigma_u \sigma_v }^+_{G,\beta}
		= \exp\biggl( \sum_{r=1}^\infty \sum_{\Loop\in \Loops_r(G^*)}
		\bigl[ w(\Loop; x')-w(\Loop; x) \bigr] \biggr),
	\]
	where $\Loops_r(G^*)$ is the collection of loops of length~$r$ in the 
	graph~$G^*$.

	Recall that we call a loop in~$G^*$ \emph{$uv$-odd} if it crosses~$\gamma$ 
	an odd number of times. Observe that for $uv$-odd loops~$\Loop$, $w(\Loop; 
	x') = -w(\Loop; x)$, while for loops~$\Loop$ that are not $uv$-odd, 
	$w(\Loop; x') = w(\Loop; x)$. It follows that
	\begin{equation}
		\label{eqn:lowTaveLoops}
		\ave{ \sigma_u \sigma_v }^+_{G,\beta}
		= \exp\biggl( -2\sum_{r=1}^\infty \sum_{\Loop\in \Loops^{uv}_r(G^*)} 
		w(\Loop; x)\biggr),
	\end{equation}
	where $\Loops^{uv}_r(G^*)$ is the collection of $uv$-odd loops of 
	length~$r$ in the graph~$G^*$.

	Note that a $uv$-odd loop of length~$r$ cannot travel far from $u$ 
	and~$v$. To be precise, these loops must be contained in $B^u_r \cup 
	B^v_r$, where $B^u_r$ is a square in the plane of side length~$r$ centred 
	at~$u$, and $B^v_r$ is defined similarly. To study the convergence 
	of~\eqref{eqn:lowTaveLoops} as $G\to\Z^2$, for arbitrary rectangles~$R$ 
	in~$\Z^2$ that can be finite or infinite, and even equal to~$\Z^2$, we now 
	define
	\[
		a_r(R^*;x) := \sum_{\Loop\in \Loops^{uv}_r(R^*)} w (\Loop; x).
	\]
	This definition makes sense both for finite and infinite~$R$, since the 
	loops contributing to the sum must be contained in $B^u_r \cup B^v_r$.
	
	Let $B^{uv}_r$ denote the smallest rectangle in~$\R^2$ containing both 
	$B^u_r$ and~$B^v_r$, and write $R^*\cap B^{uv}_r$ for the largest subgraph 
	of~$R^*$ which is a rectangle in~$\Z^{2*}$ entirely contained 
	in~$B^{uv}_r$. Then for all~$R$,
	\begin{equation}
		\label{eqn:arconv}
		a_r(R^*;x)
		= a_r(R^*\cap B^{uv}_r;x)
		= \frac12 \sum_{\Loop\in \Loops_r(R^*\cap B^{uv}_r)} \bigl[ w(\Loop; 
		x) - w(\Loop; x') \bigr].
	\end{equation}
	Now, since the volume of~$B^{uv}_r$ is bounded from above by 
	$(\norm{u-v}+r)^2$, and $\exp(2\beta_c) = \sqrt2+1$ by~\eqref{eqn:beta_c}, 
	Theorem~\ref{thm:keybound} yields the uniform bound
	\begin{equation}
		\label{eqn:arbound}
		\abs{a_r(R^*;x)}
		\leq 2(\norm{u-v}+r)^2 r^{-1} \exp\bigl( -2(\beta-\beta_c)r \bigr)
		\qquad \text{for all~$R$}.
	\end{equation}

	We now return to~\eqref{eqn:lowTaveLoops}. Since eventually, $G^*\cap 
	B^{uv}_r = \Z^{2*}\cap B^{uv}_r$ when $G\to\Z^2$, from~\eqref{eqn:arconv} 
	we conclude that
	\[
		a_r(G^*; x) \to a_r\bigl( \Z^{2*}; x \bigr)
		\qquad \text{for all $r\geq1$}.
	\]
	Moreover, the $a_r(G^*; x)$ are uniformly bounded in~$G$ by the right-hand 
	side of~\eqref{eqn:arbound}, which is summable over~$r$. Therefore, by 
	dominated convergence,
	\[
		\lim_{G\to\Z^2} \sum_{r=1}^\infty a_r(G^*; x)
		= \sum_{r=1}^\infty a_r\bigl( \Z^{2*}; x \bigr),
	\]
	where the series on the right is absolutely summable. 
	Using~\eqref{eqn:lowTaveLoops}, this proves the convergence of~$\ave{ 
	\sigma_u\sigma_v }^+_{G,\beta}$ in Theorem~\ref{thm:lowTcorr}.

	Next, we consider $\ave{ \sigma_u }^+_{G,\beta}$ for $u\in G\setminus 
	\boundary{G}$. We can treat this like $\ave{ \sigma_u\sigma_v }^+_{G, 
	\beta}$ by taking $v$ on the boundary of~$G$, since then $\sigma_v = +1$. 
	We now call a loop which crosses~$\gamma$ an odd number of times 
	\emph{$u$-odd}, since this depends only on~$u$, not~$v$. The 
	box~$B^{uv}_r$ can be replaced by~$B^u_r$ in the argument, which replaces 
	$(\norm{u-v}+r)^2$ by~$r^2$ in~\eqref{eqn:arbound}. This completes the 
	proof of Theorem~\ref{thm:lowTcorr}.
\end{proof}

\begin{proof}[Proof of Corollary~\ref{cor:lowTcorr}]
	The limit $\ave{ \sigma_u }^+_{\Z^2,\beta}$ in Theorem~\ref{thm:lowTcorr} 
	is easily seen to be independent of the choice of~$u$, and we take $u=o$ 
	as the canonical choice. We now consider what happens to the two-point 
	function when we take $u$ and~$v$ further and further apart. When $ r < 
	\norm{u-v} / 2$, the boxes $B_r^u$ and~$B_r^v$ in the proof of 
	Theorem~\ref{thm:lowTcorr} above are disjoint. If this is the case, a 
	$uv$-odd loop of length~$r$ in~$\Z^{2*}$ must be contained in either 
	$B_r^u$ or~$B_r^v$. Hence,
	\begin{multline*}
		\ave{ \sigma_u\sigma_v }^+_{\Z^2,\beta}
		= \exp\biggl( -2\sum_{r\geq\norm{u-v}/2} a_r\bigl( \Z^{2*};x \bigr) 
		\biggr) \\
		\times \exp\biggl( -2\sum_{r<\norm{u-v}/2} \bigl[ a_r(\Z^{2*}\cap 
		B^u_r;x) + a_r\bigl( \Z^{2*}\cap B^v_r; x \bigr) \bigr] \biggr).
	\end{multline*}
	When $\norm{u-v}\to \infty$, the first factor converges to~$1$ 
	exponentially fast, since the uniform bound in~\eqref{eqn:arbound} applies 
	to~$a_r\bigl( \Z^{2*}; x \bigr)$. In the second factor, the first term in 
	the sum is a sum over the $u$-odd loops of length~$r$, and the second term 
	is a sum over the $v$-odd loops. Hence the second factor factorizes and 
	converges (exponentially fast) to $[\ave{\sigma_o}^+_{\Z^2,\beta}]^2$.
\end{proof}

\subsection{High-temperature correlations}
\label{ssec:highT}

In this section, we discuss the Ising model on rectangles $G = (V,E)$ 
in~$\Z^2$ (which will again tend to~$\Z^2$) with free boundary conditions. 
From the definitions \eqref{eqn:PIsing}, \eqref{eqn:ZIsing} 
and~\eqref{eqn:aveIsing}, we have
\[
	\ave{\sigma_u \sigma_v}^{\free}_{G, \beta}
	= \sum_{\sigma\in \Omega_G^\free} \hspace{-.5em} \sigma_u \sigma_v 
	P^\free_{G,\beta}(\sigma)
	= \frac1{Z^\free_{G,\beta}} \sum_{\sigma\in \Omega_G^\free} \hspace{-.5em} 
	\sigma_u\sigma_v \prod_{xy\in E} e^{\beta \sigma_x\sigma_y}.
\]
Performing the high-temperature expansion on the right-hand side of this 
expression, in the way explained in Section~\ref{ssec:expansions}, leads to
\[
	\ave{\sigma_u \sigma_v}^{\free}_{G, \beta}
	= \frac{2^{\card{V}}(\cosh\beta)^{\card{E}}}{Z^\free_{G,\beta}} 
	\sum_{\substack{F\subset E\colon\\ \oddset{F} = \{u,v\}}} \prod_{e \in F} 
	x_e,
\]
where $x_e = \tanh\beta$ for every edge~$e$ in~$\Z^2$, and $\oddset{F}$ 
denotes the set of all vertices that have odd degree in~$(V,F)$. 
Using~\eqref{eqn:highTexpansion}, we conclude that
\begin{equation}
	\label{eqn:highTave}
	\ave{\sigma_u\sigma_v}^\free_{G, \beta}
	= \frac1{Z_G(x)} \sum_{\substack{F\subset E\colon\\ \oddset{F} = \{u,v\}}} 
	\prod_{e\in F} x_e,
\end{equation}
where $Z_G(x)$ is the graph generating function for the graph~$G$ with edge 
weight vector $x = (x_e)_{e\in E}$.

\begin{proof}[Proof of Theorem~\ref{thm:highTcorr}]
	Fix $u,v\in \Z^2$, $u\neq v$, and recall the definitions of $u^*$, $v^*$, 
	the path~$\gamma$ and the additional edges~$E_\gamma$ and 
	vertices~$V_\gamma$ from Section~\ref{ssec:isingmodel} (see 
	Figure~\ref{fig:2PointGammas}, right). For an arbitrary rectangle~$R$ 
	in~$\Z^2$ (either finite or infinite) containing $u$ and~$v$, we denote 
	by~$R_\gamma$ the graph obtained from~$R$ by adding all vertices 
	in~$V_\gamma$ to its vertex set, and all edges in~$E_\gamma$ to its edge 
	set. In~$R_\gamma$, all edges added from the set~$E_\gamma$ are considered 
	as additional, and the edges from~$R$ are considered as representative.
	
	As in the low-temperature case, we now define weights~$x'_e$ on the edge 
	set of~$\Z^2$ such that $x_e' = -x_e$ if $e$ crosses~$\gamma$, and $x_e' = 
	x_e$ otherwise. We also define edge weights $x'_\gamma(t)_e$ on the edge 
	set of~$\Z^2_\gamma$, as follows:
	\[
		x'_\gamma(t)_e = \begin{cases}
			x'_e	& \text{if $e$ is an edge of~$\Z^2$}; \\
			1		& \text{if $e \in E_\gamma\setminus \{uu^*\}$}; \\
			t		& \text{if $e = uu^*$}.
		\end{cases}
	\]
	To motivate this definition, consider a given rectangle $G = (V,E)$ 
	in~$\Z^2$, large enough so that $u,v\in V$. We claim that
	\begin{equation}
		\label{eqn:highTclaim}
		\sum_{\substack{F\subset E\colon\\ \oddset{F} = \{u,v\}}} \prod_{e\in 
		F} x_e
		= \sum_{\substack{\text{even }F\subset E\cup E_\gamma\colon \\ 
		F\supset E_\gamma }} (-1)^{C_{F}} \prod_{e\in F} x'_\gamma(1)_e.
	\end{equation}
	To see this, note that we can bijectively map every~$F$ contributing to 
	the first sum to a subgraph in the second sum, by taking the union $F\cup 
	E_\gamma$. Doing this may introduce edge crossings, whence the factor 
	$(-1)^{C_{F}}$, but these are compensated by switching from the edge 
	weight vector~$x$ to~$x_\gamma'(1)$.
	
	The crucial step is now to recognize the last expression as the derivative 
	of a graph generating function. Indeed, a simple consideration shows that
	\[
		\sum_{\substack{\text{even }F\subset E\cup E_\gamma\colon \\ F \supset 
		E_\gamma }} (-1)^{C_F} \prod_{e\in F} x'_\gamma(1)_e
		= \frac{\partial}{\partial t} \biggl( \sum_{\text{even } F\subset 
		E\cup E_\gamma} (-1)^{C_F} \prod_{e\in F} x'_\gamma(t)_e \biggr),
	\]
	evaluated at any~$t$, since any even $F\subset E_\gamma$ contributes at 
	most one factor~$t$ to the product of edge weights on the right. In 
	particular, we are allowed to evaluate the derivative at $t=0$. By 
	\eqref{eqn:highTave} and~\eqref{eqn:highTclaim}, this establishes that
	\begin{equation}
		\label{eqn:highTaveDeriv}
		 Z_G(x) \cdot \ave{\sigma_u \sigma_v}^\free_{G,\beta}
		 = \frac{\partial}{\partial t} Z_{G_\gamma}\bigl( x_\gamma'(t) \bigr) 
		 \Big|_{t=0}.
	\end{equation}

	We now fix $\beta\in (0,\beta_c)$, so that by~\eqref{eqn:beta_c}, $x_e\in 
	(0,\sqrt2-1)$ for every $e\in E$, and by Theorem~\ref{thm:keybound}, the 
	spectral radius of~$\Lambda(x')$ is strictly less than~$1$. We now need a 
	similar bound on the spectral radius of the matrix~$\Lambda_\gamma\bigl( 
	x'_\gamma(t) \bigr)$, which is the transition matrix for the modified 
	graph~$G_\gamma$ with edge weight vector~$x'_\gamma(t)$. The difference 
	between these two matrices is that $\Lambda_\gamma\bigl( x'_\gamma(t) 
	\bigr)$ allows transitions between $u$ and~$v$ along the chain of 
	additional edges in~$E_\gamma$. This means that the 32~matrix entries from 
	$du$ to~$vd'$ and from $dv$ to~$ud'$, with $d,d'\in \{\north, \south, 
	\east, \west\}$, are nonzero in~$\Lambda_\gamma \bigl( x'_\gamma(t) 
	\bigr)$ for $t\neq0$, while they are~$0$ in~$\Lambda(x')$; all other 
	entries of the two matrices are the same.

	The 32~deviating matrix entries are all of the form $te^{i\phi/2}$, where 
	$\phi$ is a sum of turning angles. Here, $t$ will be treated as a complex 
	variable. For $t=0$, $\Lambda_\gamma\bigl( x'_\gamma(t) \bigr) = 
	\Lambda(x')$. Since the eigenvalues vary continuously with~$t$, we 
	conclude that there exists $\eps>0$ such that for all~$t$ satisfying 
	$\abs{t} < \eps$, the spectral radius of $\Lambda_\gamma\bigl( 
	x'_\gamma(t) \bigr)$ is bounded from above by some $\alpha\in (0,1)$.
	
	Hence, if $\abs{t} < \eps$, Theorem~\ref{thm:Zdeterminant} applies, and we 
	obtain
	\[
		Z_{G_\gamma}\bigl( x_\gamma'(t) \bigr)
		= \exp\biggl( \sum_{r=1}^\infty f_{\gamma r}(t) \biggr),
	\]
	where
	\[
		f_{\gamma r}(t) = \sum_{\Loop\in \Loops_r(G_\gamma)} w\bigl( \Loop; 
		x_\gamma'(t) \bigr).
	\]
	Note that, this last sum being finite, the $f_{\gamma r}(t)$ are 
	polynomials in~$t$. Also, from~\eqref{eqn:feigenvals} it follows that 
	$\abs{f_{\gamma r}(t)} \leq 2\card{V} \alpha^r$. Therefore, the partial 
	sums of the series $\sum_r f_{\gamma r}(t)$ are uniformly convergent for 
	$|t|<\eps$, and the sum of the series is an analytic function of~$t$. 
	Moreover, the derivatives of the partial sums also converge uniformly to 
	the derivative of the sum of the series.

	From all this, it follows that the right-hand side 
	of~\eqref{eqn:highTaveDeriv} is equal to
	\[
		\Biggl( \sum_{r=1}^\infty \frac{\partial}{\partial t} \sum_{\Loop\in 
		\Loops_r(G_\gamma)} w\bigl( \Loop; x_{\gamma}'(t) \bigr) \Big|_{t=0} 
		\Biggr) \exp\Biggl( \sum_{r=1}^\infty \sum_{\Loop\in 
		\Loops_r(G_\gamma)} w\bigl( \Loop; x_{\gamma}'(0) \bigr) \Biggr).
	\]
	In the first factor, the only loops that survive the differentiation are 
	those that visit the edge~$uu^*$, since only they contribute a factor~$t$ 
	to the weight. Taking the derivative at $t=0$, we are only left with those 
	loops that visit the edge~$uu^*$ exactly once. In the second factor, 
	because we set $t=0$, the only loops that contribute are those that do not 
	visit~$uu^*$. This leaves precisely all loops in the graph~$G$. The 
	right-hand side of~\eqref{eqn:highTaveDeriv} therefore becomes
	\[
		\Biggl( \sum_{r=1}^\infty \sum_{\Loop \in \Loops^{uu^*}_r(G_\gamma)} 
		w\bigl( \Loop; x_\gamma'(1) \bigr) \Biggr) \exp\Biggl( 
		\sum_{r=1}^\infty \sum_{\Loop\in \Loops_r(G)} w(\Loop; x') \Biggr),
	\]
	where $\Loops^{uu^*}_r(G_\gamma)$ is the set of loops of length~$r$ 
	in~$G_\gamma$ that visit~$uu^*$ once. From this, applying 
	Theorem~\ref{thm:Zdeterminant} again to the second factor, we find that
	\begin{equation}
		\label{eqn:highTaveLoops}
		\ave{\sigma_u \sigma_v}^\free_{G,\beta}
		= \Biggl( \sum_{r=1}^\infty \sum_{\Loop\in \Loops^{uu^*}_r(G_\gamma)} 
		w\bigl( \Loop; x_{\gamma}'(1) \bigr) \Biggr) \frac{Z_G(x')}{Z_G(x)}.
	\end{equation}

	Note the ratio of graph generating functions in~\eqref{eqn:highTaveLoops}. 
	Recall that we have seen such a ratio of graph generating functions before 
	in the low-temperature case, namely in~\eqref{eqn:lowTaveRatio}. Thus, 
	this ratio can be interpreted as a two-point function between the spins at 
	$u^*$ and~$v^*$ in a dual Ising model with positive boundary conditions at 
	the dual low temperature~$\beta^*$, given by $\exp(-2\beta^*) = 
	\tanh\beta$. Using~\eqref{eqn:lowTaveLoops}, we can express this ratio in 
	terms of a sum over all $u^*v^*$-odd loops in the graph~$G$, if we like.

	Next, we want to consider the limit as $G\to \Z^2$. By the argument given 
	in Section~\ref{ssec:lowT}, we already know that the ratio of graph 
	generating functions in~\eqref{eqn:highTaveLoops} converges to 
	$\ave{\sigma_{u^*} \sigma_{v^*}}^+_{ \Z^{2*}, \beta^* }$. It remains to 
	consider what happens to the sum over the loops that visit~$uu^*$ once. To 
	this end, for a general finite or infinite rectangle~$R$ in~$\Z^2$ 
	containing $u$ and~$v$, we define
	\begin{equation}
		\label{eqn:argamma}
		a_r(R_\gamma; x'_\gamma) := \sum_{\Loop\in \Loops^{uu^*}_r(R_\gamma)} 
		w(\Loop; x'_\gamma),
	\end{equation}
	where we have simplified the notation by letting $x'_\gamma \equiv 
	x'_\gamma(1)$.
	
	As in the low-temperature case, the loops that contribute to 
	$a_r(R_\gamma; x'_\gamma)$ must be confined to the box~$B^{uv}_r$, defined 
	in the same way as before, except possibly for the steps taken along the 
	additional edges in~$E_\gamma$, which are not counted in the length of the 
	loop, and are allowed to go outside~$B^{uv}_r$. Hence,
	\begin{equation}
		\label{eqn:argammaeq}
		a_r(R_\gamma; x'_\gamma)
		= a_r\bigl( (R \cap B^{uv}_r)_\gamma; x'_\gamma \bigr),
	\end{equation}
	where $R \cap B^{uv}_r$ is the largest subgraph of~$R$ contained 
	in~$B^{uv}_r$, as before. It follows that $a_r(G_\gamma; x'_\gamma) \to 
	a_r\bigl( \Z^2_\gamma; x'_\gamma \bigr)$ for all $r\geq1$. As before, we 
	now want to use dominated convergence to prove that
	\begin{equation}
		\label{eqn:sumargamma}
		\lim_{G\to\Z^2} \sum_{r=1}^\infty a_r(G_\gamma; x'_\gamma)
		= \sum_{r=1}^\infty  a_r\bigl( \Z^2_\gamma; x'_\gamma \bigr),
	\end{equation}
	and that the right-hand side is absolutely summable. This requires an 
	appropriate uniform bound (in~$R$) on the right-hand side 
	of~\eqref{eqn:argamma}.

	To obtain this bound, by~\eqref{eqn:argammaeq} it is sufficient to 
	consider an arbitrary finite rectangle~$R$ in~$\Z^2$ containing $u$ 
	and~$v$, and such that~$R$ is contained in~$B^{uv}_r$. Note that every 
	loop of length~$r$ in~$R_\gamma$ which visits the edge~$uu^*$ once, has a 
	representation of the form $\Path \concat \gamma$, where $\Path$ is a path 
	of length~$r$ in~$R$ from $v$ to~$u$. Here, we use the facts that the 
	part~$\Path$ of the loop never visits~$uu^*$, and that the steps taken 
	along~$\gamma$ from $u$ to~$v$ do not contribute towards the length of the 
	loop, since they are along additional edges.
	
	Let $\Lambda_R(x')$ be the transition matrix for the graph~$R$ with edge 
	weights~$x'_e$. Note that the sum of the weights of all paths~$\Path$ of 
	length~$r$ from $v\east$ to~$\north u$, for instance, is given by the 
	entry of the matrix~$\Lambda_R^r(x')$ in row~$v\east$ and column~$\north 
	u$. To compute the sum of the weights of the corresponding loops $\Path 
	\concat \gamma$, we only need to multiply this entry by the 
	factor~$e^{i\phi/2}$, where $\phi$ is the sum of the turning angles 
	encountered in the path from $\north u$ to~$v\east$ along the edges 
	in~$E_\gamma$. From these observations, we can conclude that
	\[
		a_r(R_\gamma; x'_\gamma) \leq 16\, \norm{\Lambda_R^r(x')}_{\max},
	\]
	where $\norm{\cdot}_{\max}$ denotes the maximum-entry norm, and the 
	factor~16 comes from the fact that there are 4~directed (representative) 
	edges pointing out from~$v$, and 4 pointing to~$u$. By  
	Theorem~\ref{thm:keybound} and the fact that the maximum-entry norm of a 
	matrix is bounded by the operator norm, and using that 
	$(\tanh\beta_c)^{-1} = \sqrt2+1$ by~\eqref{eqn:beta_c}, we obtain
	\begin{equation}
		\label{eqn:argammaBound}
		a_r\bigl( R_\gamma; x'_\gamma \bigr)
		\leq 16\, \norm{\Lambda_R^r(x')}
		\leq 16\, \norm{\Lambda_R(x')}^r
		\leq 16\, \Bigl( \frac{\tanh\beta}{\tanh\beta_c} \Bigr)^r.
	\end{equation}
	
	We emphasize that this bound holds uniformly for all finite and infinite 
	rectangles~$R$ containing $u$ and~$v$. Hence, \eqref{eqn:sumargamma} holds 
	by dominated convergence, and this completes the proof of 
	Theorem~\ref{thm:highTcorr}.
\end{proof}

\begin{proof}[Proof of Corollary~\ref{cor:highTcorr}]
	We now consider what happens to the two-point function studied above when 
	we let $\norm{u-v}$ tend to infinity. Since the loops that visit~$uu^*$ 
	necessarily have length at least $\norm{u-v}$, we can write
	\[
		\ave{\sigma_u \sigma_v}^\free_{\Z^2,\beta}
		= \biggl( \sum_{r\geq \norm{u-v}} a_r\bigl( \Z^2_\gamma; x'_\gamma 
		\bigr) \biggr) \ave{\sigma_{u^*} \sigma_{v^*}}^+_{\Z^{2*}, \beta^*}.
	\]
	By Theorem~\ref{thm:lowTcorr}, the two-point function on the right is 
	bounded between $0$ and~$1$. Alternatively, at this stage we could also 
	observe that the ratio of graph generating functions 
	in~\eqref{eqn:highTaveLoops} is always between $-1$ and~$+1$, since the 
	same even subgraphs contribute to both generating functions, but only in 
	the numerator, some of them come with a negative sign. Furthermore, the 
	bound in~\eqref{eqn:argammaBound} holds for $a_r\bigl( \Z^2_\gamma; 
	x'_\gamma \bigr)$. This gives the desired upper bound. That the two-point 
	function is nonnegative follows directly from~\eqref{eqn:highTave}.
\end{proof}

\bibliography{ising2d}

\end{document}